\documentclass{article}
\usepackage{graphicx} 
\usepackage[a4paper, margin=3 cm]{geometry}
\usepackage{hyperref}
\usepackage{amsthm}
\usepackage{thm-restate}
\usepackage{geometry}
\usepackage{amssymb}
\usepackage{amsmath}
\usepackage{cleveref}
\usepackage{comment}
\usepackage{setspace}
\usepackage{xcolor}
\newtheorem{thm}{Theorem}[section]
\newtheorem{conjecture}[thm]{Conjecture}
\newtheorem{obs}[thm]{Observation}
\newtheorem{cor}[thm]{Corollary}
\newtheorem{problem}[thm]{Problem}
\newtheorem{prop}[thm]{Proposition}
\newtheorem{lemma}[thm]{Lemma}

\newtheorem{define}[thm]{Definition}

\usepackage{cite}
\usepackage{mathrsfs}
\usepackage{tikz}
\usepackage{caption}
\usepackage{subcaption}
\usetikzlibrary{shadows.blur}
\usetikzlibrary{decorations.pathmorphing}
\usetikzlibrary{arrows.meta}

\renewcommand{\mathscr}{\mathcal}

\def\Aut{\mathrm{Aut}}

\begin{document}
\title{Intersections of longest cycles in vertex-transitive and highly connected graphs}

\author{Jie Ma\thanks{School of Mathematical Sciences, University of Science and Technology of China, Hefei, Anhui 230026, China, and Yau Mathematical Sciences Center, Tsinghua University, Beijing 100084, China.
Research supported by National Key Research and Development Program of China 2023YFA1010201 and National Natural Science Foundation of China grant 12125106. Email: {\tt jiema@ustc.edu.cn}}
\and Ziyuan Zhao\thanks{School of Mathematical Sciences, University of Science and Technology of China, Hefei, Anhui 230026, China. Research supported by Innovation Program for Quantum Science and Technology 2021ZD0302902. Email: {\tt zyzhao2024@mail.ustc.edu.cn}}
}
\date{\today}
\maketitle

\begin{abstract}
Motivated by the classical conjectures of Lov\'asz, Thomassen, and Smith, recent work has renewed interest in the study of longest cycles in important graph families, such as vertex-transitive and highly connected graphs.
In particular, Groenland et al.\ proved that if two longest cycles and in a graph share $m$ vertices, then there exists a vertex cut of size $O(m^{8/5})$ separating them, yielding improved bounds toward these conjectures. Their proof combines Tur\'an-type arguments with computer-assisted search.

We prove two results addressing problems of Babai (1979) and Smith (1984) on intersections of longest cycles in vertex-transitive and highly connected graphs. First, we strengthen the bound of Groenland et al.\ by showing that if two longest cycles and in a graph share $m$ vertices, then there exists a vertex cut of size $O(m^{3/2})$ separating them. As a consequence, we show that in every \(k\)-connected graph, any two longest cycles intersect in at least \(\Omega(k^{2/3})\) vertices, improving the best known bound toward Smith's conjecture. Our proof is purely combinatorial, employing supersaturation-type estimates beyond the existing Tur\'an-type approach.  
Second, we prove that in every connected vertex-transitive graph on \(n\) vertices, any two longest cycles intersect in at least \(f(n)\) vertices for some function \(f(n)\to\infty\) as \(n\to\infty\), thereby resolving a problem of Babai (1979) for the class of vertex-transitive graphs central to his original motivation. In doing so, we introduce a new method for constructing longer cycles in vertex-transitive graphs based on a given cycle, which may be of independent interest.
\end{abstract}

\section{Introduction}
The study of longest cycles is a central theme in both extremal and structural graph theory. Our work is motivated by several longstanding problems in this area, including two conjectures of Lovász \cite{lovasz1969combinatorial} and Thomassen (cf. \cite{barat2010his}) on vertex-transitive graphs, a conjecture of Smith (cf. \cite{bondy1996basic}) on highly connected graphs, and a problem of Babai \cite{babai1979long}, originally motivated by vertex-transitive graphs but posed implicitly for 3-connected cubic graphs. The first two conjectures concern Hamiltonicity, while the latter two address the intersection sizes of longest cycles.
Our main results make progress toward these problems and highlight closer connections between them.
Throughout, we denote by $c(G)$ the length of a longest cycle in the graph $G$.

\subsection{Background}
A graph $G$ is said to be {\it vertex-transitive} if, for any two vertices $u,v \in V(G)$, there exists an automorphism of $G$ that maps $u$ to $v$.
The classical conjectures of Lov\'asz~\cite{lovasz1969combinatorial} and Thomassen (cf.~\cite{barat2010his}) assert that every connected vertex-transitive graph on \( n \ge 3 \) vertices contains a Hamiltonian cycle (resp.\ path), with only finitely many exceptions. These conjectures remain widely open despite extensive study (see the survey \cite{KM09}, more recent papers\cite{du2021resolving,christofides2014hamilton}, and the references therein). 
In 1979, Babai~\cite{babai1979long} proved that every connected vertex-transitive graph \( G \) on \( n \geq 3 \) vertices contains a cycle of length at least \( \Omega(\sqrt{n}) \). His elegant proof proceeds by double-counting the number \( T \) of triples \( \{C, D, v\} \), where \( C \) and \( D \) are distinct longest cycles in \( G \) and \( v \in V(C) \cap V(D) \) is a common vertex.
Since \( G \) is vertex-transitive, every vertex lies in exactly \( d \geq 1 \) longest cycles, so \( G \) contains exactly \( nd/c(G) \) distinct longest cycles.
We may assume that $G$ is 3-connected.\footnote{Unless $G$ is 2-regular (see~\cite{godsil2013algebraic}), in which case it is simply the cycle \( C_n \) and the result of Babai holds trivially.} 
Since any two longest cycles in every 3-connected graph intersect in at least three vertices,\footnote{We will elaborate on this fact and its generalizations shortly; see Conjecture~\ref{conj:Smith}.} we have
\[
n \binom{d}{2} = T \ge 3 \cdot \binom{nd/c(G)}{2},
\]
which implies that \( c(G) \ge \sqrt{3n} \).
Observe that if any two longest cycles in such a graph \( G \) intersect in at least \( f(n) \) vertices, then one immediately obtains the stronger bound \( c(G) \geq \sqrt{n\cdot f(n)} \).
Motivated by this observation, Babai posed the following problem in the same paper~\cite{babai1979long}, restricting attention to 3-connected cubic graphs as an initial case.

\begin{problem}[Babai~\cite{babai1979long}, Problem~2]\label{prob:Babai}
Is there a function \( f(n) \to \infty \) as \( n \to \infty \) such that in every 3-connected cubic graph on \( n \) vertices, any two longest cycles intersect in at least \( f(n) \) vertices?
\end{problem}

\noindent It is therefore natural to ask the analogous question for all connected vertex-transitive graphs.

\begin{problem}\label{prob:vert-tran}
Is there a function \( f(n) \to \infty \) as \( n \to \infty \) such that in every connected vertex-transitive graph on \( n \) vertices, any two longest cycles intersect in at least \( f(n) \) vertices?
\end{problem}

The $\Omega(n^{1/2})$ bound on the length of the longest cycle in connected vertex-transitive graphs stood for over four decades, until DeVos~\cite{devos2023longer} improved it in 2023.  
Using the Orbit-Stabilizer Theorem, 
DeVos~\cite{devos2023longer} established a notable double-counting lemma relating the problem to the {\it transversal} of longest cycles, a concept studied extensively (see \cite{Gallai68,kierstead2023improved,long2021sublinear,rautenbach2014transversals}); a subset of vertices in a graph $G$ is called a {\it $t$-transversal} if it contains at least $t$ vertices from every longest cycle in $G$.  
His lemma shows that if a connected vertex-transitive graph $G$ on $n$ vertices contains a $t$-transversal $A$, then
\begin{align}\label{equ:DeVos}
 c(G) \geq tn/|A|.   
\end{align}
Combining this with the simple fact that if two longest cycles and in a graph share $m$ vertices, then there exists a vertex cut of size at most $m^2$ separating them (which also forms a $1$-transversal; see Proposition~\ref{prop:intersect}), 
DeVos~\cite{devos2023longer} showed that $c(G) = \Omega(n^{3/5})$
for every connected vertex-transitive graph $G$ on \(n\) vertices.
Recently, Groenland, Longbrake, Steiner, Turcotte, and Yepremyan~\cite{groenland2024longest} gave a unified treatment concerning longest cycles in vertex-transitive and highly connected graphs. 
Their work extends DeVos' approach and shows that for any connected vertex-transitive graph $G$,
$c(G) = \Omega\bigl(n^{13/21}\bigr).$
The central result underlying their unified approach is the following, whose proof combines Tur\'an-type arguments with computer-assisted search and linear programming.
\begin{thm}[Groenland et al.~\cite{groenland2024longest}, Lemma 1.4]\label{thm:GLSTY}
If two longest cycles in a graph $G$ share $m$ vertices, then there exists a vertex cut of size $O(m^{8/5})$ separating these two cycles.
\end{thm}
\noindent More recently, Norin, Steiner, Thomass\'e, and Wollan~\cite{norin2025small} achieved a further improvement, showing that 
$c(G) = \Omega\bigl(n^{9/14}\bigr)$
for every connected vertex-transitive graph $G$ on $n$ vertices. 
Using equation \eqref{equ:DeVos}, their approach is to upper-bound the size of certain $1$-transversals \(A\) as a function of \(c(G)\), in contrast to \cite{devos2023longer,groenland2024longest}, where the bound is expressed in terms of the intersection size of two longest cycles (cf. Theorem~\ref{thm:GLSTY}).

Another motivating problem is a longstanding conjecture on the intersection size of two longest cycles in highly connected graphs, attributed to S.~Smith and traceable to at least 1984. 
As noted earlier in Babai's work~\cite{babai1979long}, such insights can be valuable for addressing other problems on cycles.

\begin{conjecture}[Smith's Conjecture; see, e.g., Conjecture~4.15 in \cite{bondy1996basic}]\label{conj:Smith}
    For every \( k \ge 2 \), any two longest cycles in a \( k \)-connected graph intersect in at least \( k \) vertices. 
\end{conjecture}    

\noindent 
This conjecture has attracted considerable attention (see the survey~\cite{shabbir2013intersecting}). 
It has been confirmed for \( k \le 8 \) in~\cite{grotschel1984intersections,stewart1995intersections} and for \( k \ge (n+16)/7 \) in~\cite{gutierrez2024two}. 
For the general case, Burr and Zamfirescu (see~\cite{bondy1996basic}) proved that any two longest cycles in a \( k \)-connected graph intersect in at least \( \sqrt{k}-1 \) vertices. 
This was later improved in 1998 by Chen, Faudree, and Gould~\cite{chen1998intersections} to \( \Omega(k^{3/5}) \) using a Tur\'an-type argument.
Recently, Groenland et al.~\cite{groenland2024longest}, based on Theorem~\ref{thm:GLSTY}, established the best-known bound to date and proved that the intersection size of any two longest cycles in a $k$-connected graph is \( \Omega(k^{5/8}) \).

\subsection{Main results}
In this paper, motivated by the aforementioned conjectures and problems, we present two results on the intersection size of longest cycles in vertex-transitive and highly connected graphs. 

Our first result strengthens Theorem~\ref{thm:GLSTY} in the following quantitative form.

\begin{thm}\label{thm:lct upp}
If two longest cycles in a graph $G$ share $m$ vertices, then there exists a vertex cut of size $O(m^\frac{3}{2})$ separating these two cycles.\footnote{We emphasize again that such a vertex cut forms a $1$-transversal in \(G\); see Proposition~\ref{prop:intersect}.}
\end{thm}

\noindent Our proof is purely combinatorial and relies on supersaturation-type estimates, extending beyond the existing Tur\'an-type approaches (e.g., \cite{chen1998intersections,groenland2024longest}).
Using this result, Menger's Theorem immediately yields the following bound for Conjecture~\ref{conj:Smith}, improving upon the current record, namely the \( \Omega(k^{5/8}) \) bound of Groenland et al.~\cite{groenland2024longest}.

\begin{cor}\label{cor:smith}
Let $k\geq 2$ be any integer. Every two longest cycles in a $k$-connected graph intersect in $\Omega(k^{2/3})$ vertices.  
\end{cor}

\noindent 
Another implication of Theorem~\ref{thm:lct upp} concerns the length \( c(G) \) of a longest cycle in a connected vertex-transitive graph \( G \) on \( n \) vertices.
To see this, let \(m\) denote the minimum size of the intersection of two longest cycles in \(G\). 
Then, evidently the vertex set of every longest cycle forms an \(m\)-transversal, and by Theorem~\ref{thm:lct upp}, there exists a \(1\)-transversal of size \(O(m^{3/2})\). 
Applying equation~\eqref{equ:DeVos} twice, we obtain
\[
c(G) \ge \min_{m \in \mathbb{N}} \max\left\{ \sqrt{mn}, \frac{n}{O(m^{3/2})} \right\} = \Omega(n^{5/8}).
\]
This improves the \( \Omega(n^{13/21}) \) bound of Groenland et al.~\cite{groenland2024longest}, but falls short of the very recent \( \Omega(n^{9/14}) \) bound of Norin et al.~\cite{norin2025small}. 
On the other hand, since every connected \(d\)-regular vertex-transitive graph is \(\Omega(d)\)-connected (see~\cite{godsil2013algebraic}), Theorem~\ref{thm:lct upp} outperforms the \( \Omega(n^{9/14}) \) bound for connected \(d\)-regular vertex-transitive graphs when \(d = \Omega(n^{3/7})\), as follows.

\begin{cor}\label{cor:vert-tran-c(G)}
Let $n>d\geq 2$. Every connected $d$-regular vertex-transitive graph on $n$ vertices contains a cycle of length $\Omega(d^\frac{1}{3}n^\frac{1}{2})$.
\end{cor}

Our second result is motivated by Babai's work~\cite{babai1979long}. 
We prove the following, showing that any two longest cycles in a connected vertex-transitive graph intersect in sufficiently many vertices.

\begin{thm}\label{thm:transitive babai}
Let $n>d\geq 2$.
Every two longest cycles in a connected $d$-regular vertex-transitive graph on $n$ vertices intersect in $\Omega\big((\log_d n)^{1/3}\big)$ vertices.
\end{thm}

\noindent Combining this with Corollary~\ref{cor:smith}, we obtain the following corollary, which provides an affirmative answer to Problem~\ref{prob:vert-tran}. 
Throughout, we use \(\ln x\) to denote the natural logarithm.

\begin{cor}\label{cor:vert-tran}
    Every two longest cycles in a connected vertex-transitive graph on $n$ vertices intersect in $\Omega\big((\ln n/\ln\ln n)^{1/3}\big)$ vertices.
\end{cor}

\noindent This can be seen as evidence (albeit arguably weak) supporting the conjecture of Lov\'asz~\cite{lovasz1969combinatorial}. 

\medskip

The rest of this paper is organized as follows. 
In Section~2, we provide the necessary preliminaries, including the crucial relation that reduces the proof of Theorem~\ref{thm:lct upp} to a problem in extremal graph theory. 
In Section~3, we present the proofs of Theorem~\ref{thm:lct upp} and Corollary~\ref{cor:vert-tran-c(G)}. 
In Section~4, we prove Theorem~\ref{thm:transitive babai} and Corollary~\ref{cor:vert-tran}. 
Finally, in Section~5, we conclude with several remarks.

\section{Preliminaries}
In this section, we first introduce the notations used throughout this paper, and then define a useful auxiliary graph based on two longest cycles for the proof of Theorem~\ref{thm:lct upp}.  
We also explain how this auxiliary graph plays a key role in the proof of Theorem~\ref{thm:lct upp}.

\subsection{Notations}
We mostly follow the standard graph-theoretic notation from \cite{diestel2024graph}.  
All graphs are finite unless otherwise specified, and the term \textit{disjoint} always refers to \textit{vertex-disjoint}.
Let \(G\) be a graph.  
As mentioned earlier, we write $c(G)$ for the length of a longest cycle in $G$.
Let \(A, B \subseteq V(G)\) be two vertex subsets, not necessarily disjoint.
A path $P=x_1\cdots x_t$ in $G$ is called an \textit{$(A,B)$-path}, if $V(P)\cap A=\{x_1\}$ and $V(P)\cap B=\{x_t\}$.
If $A=\{a\}$, we abbreviate $(\{a\},B)$-paths by $(a,B)$-paths.
For a subgraph $H$ of $G$, we abbreviate $(V(H),B)$-paths by $(H,B)$-paths.
We say that a vertex subset \(S \subseteq V(G)\) \textit{separates} \(A\) and \(B\) if every \((A,B)\)-path intersects \(S\).  
In this case, we call \(S\) an \textit{\((A,B)\)-separator}.  
Note that necessarily \(A \cap B \subseteq S\).
We also write \(G - A\) for the subgraph of \(G\) induced by \(V(G) \setminus A\).

For a path \(P\) (or a cycle \(C\)), we write \(|P|\) (or \(|C|\)) for the number of its edges.
Given \(u,v \in V(P)\), let \(P[u,v]\) denote the unique subpath of \(P\) with endpoints \(u\) and \(v\).    
If \(P\) has endpoints \(a,b\), and we fix a direction from \(a\) to \(b\), we define the associated linear order \(\prec\) on \(V(P)\) as follows: for \(u,v \in V(P)\), we write \(u \prec v\) if \(a,u,v,b\) occur on \(P\) in the given direction.

For integers \(a \le b\), we denote \([a,b] = \{\, i \in \mathbb{Z} : a \le i \le b \,\}\).  
If \(a\) is positive, we write \([a]\) for the set \([1,a]\).  

\subsection{An auxiliary graph associated with two longest cycles}
In this subsection, let \(G\) be a 2-connected graph, and let \(X, Y\) be any two longest cycles in \(G\).  
Here, we discuss some basic propositions and define an auxiliary graph arising from this setting.  

We begin with a simple fact.
Recall the definition of a \(t\)-transversal given before equation \eqref{equ:DeVos}.  
We refer to a \(1\)-transversal simply as a \emph{transversal}.

\begin{prop}\label{prop:intersect}
We have $V(X)\cap V(Y)\neq \emptyset$, and any vertex-cut separating $X$ and $Y$ is a transversal of $G$.
\end{prop}

\begin{proof}
Suppose that \(V(X) \cap V(Y) = \emptyset\).  
Since \(G\) is 2-connected, there exist two disjoint \((X,Y)\)-paths.  
Together with \(X\cup Y\), these paths would form two cycles \(X'\) and \(Y'\) satisfying \(|X'| + |Y'| > |X| + |Y|\),  
which contradicts the maximality of \(X\) and \(Y\).

Now consider any vertex-cut \(A\) separating \(X\) and \(Y\).  
Clearly, \(A\) contains all vertices in \(V(X) \cap V(Y)\), which is non-empty.  
Let \(Z\) be any other longest cycle in \(G\).  
Since \(V(Z) \cap V(X) \neq \emptyset\) and \(V(Z) \cap V(Y) \neq \emptyset\),  
there exists a subpath of \(Z\) which is a \(\big(V(Z) \cap V(X), V(Z) \cap V(Y)\big)\)-path.  
This path must intersect \(A\), implying \(V(Z) \cap A \neq \emptyset\).  
Hence, \(A\) is a transversal of \(G\).
\end{proof}

Let \(M = V(X) \cap V(Y)\) and \(m = |M|\).  
Then \(X - M\) and \(Y - M\) decompose into \(m\) subpaths \(X_1, \ldots, X_m\) and \(Y_1, \ldots, Y_m\), respectively (allowing empty subpaths).  
In the remainder of the paper, we refer to these subpaths $X_i$ or $Y_i$ as {\it segments}.
The following proposition establishes a useful property for any collection of disjoint \((X - M, Y - M)\)-paths in \(G\).

\begin{prop}\label{prop:X-Y path}
    For any $1 \leq i, j \leq m$, there do not exist two disjoint $(X-M, Y-M)$-paths in $G$ whose endpoints belong to $X_i$ and $Y_j$.
\end{prop}

\begin{proof}
    Assume to the contrary that such paths $L_1$ and $L_2$ exist.
    Let $V(L_t)\cap V(X_i)=\{u_t\}$ and $V(L_t)\cap V(Y_j)=\{v_t\}$ for $t\in [2]$.
    Without loss of generality, we can assume that $|X_i[u_1,u_2]|\leq |Y_j[v_1,v_2]|$.
    Replacing the subpath $X_i[u_1,u_2]$ in $X_i$ by the new $(u_1,u_2)$-path $L_1\cup Y_j[v_1,v_2]\cup L_2$, we obtain a cycle in $G$ with more than $|X|$ edges, contradicting the maximality of $X$.
\end{proof}

With these propositions in place, we now formalize the notion of the auxiliary graph associated with \(X\) and \(Y\). 
It is defined for a given collection \(\mathcal{P}\) of disjoint \((X-M,Y-M)\)-paths and will play a crucial role in the proof of Theorem~\ref{thm:lct upp}.

\begin{define}\label{def:aux F}
Let $X, Y$ be two longest cycles in a graph $G$. Let $M=V(X)\cap V(Y)$ and $m=|M|$.
Let $\{X_1,\cdots,X_m\}$ and $\{Y_1,\cdots,Y_m\}$ denote the path decompositions of $X-M$ and $Y-M$, respectively.
Consider any collection $\mathcal{P}$ of disjoint $(X-M,Y-M)$-paths in $G$.
Then the auxiliary graph $F:= F(X,Y,\mathcal{P})$ is constructed as follows:
\begin{itemize}
    \item[(1)] The graph $F$ is bipartite with two parts $A=\{x_1,\cdots,x_m\}$ and $B=\{y_1,\cdots,y_{m}\}$, where the vertex $x_i$ (resp. $y_i$)  corresponds to the segment $X_i$ (resp. $Y_i$) in $G$.
    \item[(2)] For $i,j\in [m]$, define the edge $x_iy_j\in E(F)$ if and only if $\mathcal{P}$ contains a $(X_i,Y_j)$-path.
\end{itemize}
\end{define}

We point out that by Proposition~\ref{prop:X-Y path}, the graph $F$ is simple and has exactly $2m$ vertices and $|\mathcal{P}|$ edges.
The auxiliary graph $F$ was explicitly considered in \cite{chen1998intersections,groenland2024longest}, where the main idea in both papers is to show that $F$ does not contain a fixed bipartite graph $H$ as a subgraph (that is, $F$ is {\it $H$-free}). 
Applying known estimates for the Tur\'an number of $H$ then yields an upper bound on $|\mathcal{P}| = e(F)$ as a function of $m$.
In particular, \cite{chen1998intersections} proves that $F$ is $K_{3,257}$-free, while \cite{groenland2024longest} shows that $F$ is $Q_3^+$-free, where $Q_3^+$ denotes the graph obtained from $K_{4,4}$ by deleting a matching of size three.
In the next section, we employ supersaturation arguments, extending beyond the earlier approaches that rely solely on Turán-type bounds.
Specifically, we establish Theorem~\ref{thm:lct upp} by proving that the graph $F$ contains only $O(m)$ copies of $K_{2,7}$, which implies
\(
|\mathcal{P}| = e(F) = O(m^{3/2}).
\)
Menger's Theorem then provides the required vertex cut for Theorem~\ref{thm:lct upp}.

\section{Proof of Theorem~\ref{thm:lct upp}}
In this section, we prove Theorem~\ref{thm:lct upp}.
Throughout, let $G$ be a graph and $X, Y$ be two longest cycles in $G$ with $M=V(X)\cap V(Y)$ and $m=|M|$.
Our goal is to show that there exists a vertex cut of size $O(m^{3/2})$ separating $X$ and $Y$.
This is evident if $X$ and $Y$ lie in different blocks of $G$.
Hence, we may assume that $G$ is 2-connected.

We fix a cyclic ordering for each of the cycles $X$ and $Y$.
Let $\{X_1, \dots, X_m\}$ and $\{Y_1, \dots, Y_m\}$ denote the path decompositions of $X - M$ and $Y - M$, respectively, such that $X_1, \dots, X_m$ (resp. $Y_1, \dots, Y_m$) appear sequentially according to the chosen cyclic ordering of $X$ (resp. $Y$).

Consider any collection $\mathcal{P}$ of disjoint $(X-M,Y-M)$-paths, and let $F := F(X,Y,\mathcal{P})$ be the bipartite graph defined in Definition~\ref{def:aux F}.
Our strategy is to use certain local structural conditions of $F$ to find two cycles $Q_1, Q_2$ in $G$ such that
\[
E(Q_1) \cup E(Q_2) \supseteq E(X) \cup E(Y) \quad \text{and} \quad |Q_1| + |Q_2| > |X| + |Y|,
\]
thereby contradicting the maximality of $X$ and $Y$.
We call such a pair $\{Q_1,Q_2\}$ of cycles a {\bf winning certificate}.

We now introduce a key concept -- the {\it types} of 4-cycles in $F$, determined by the relative positions in both $X$ and $Y$ of the endpoints of the four paths in $\mathcal{P}$ corresponding to its edges.
Consider a 4-cycle $x_i y_k x_j y_\ell$ in $F$, where, throughout this section, we always assume that $i < j$ and $k < \ell$.
By Definition~\ref{def:aux F}, 
the edge $x_i y_k$ in $F$ uniquely corresponds to an $(X_i,Y_k)$-path $P_{ik} \in \mathcal{P}$ in $G$;
similarly, the edges $x_i y_\ell$, $x_j y_k$, and $x_j y_\ell$ correspond to the paths $P_{i\ell}$, $P_{jk}$, and $P_{j\ell}$ in $\mathcal{P}$, respectively.
For convenience, we set
\[
\mathcal{P}' = \{P_{ik}, P_{i\ell}, P_{jk}, P_{j\ell}\}.
\]
We denote by $u_{ik}, u_{i\ell}, u_{jk}, u_{j\ell}$ the endpoints of $\mathcal{P}'$ in $X$, where $u_{ik} = V(X) \cap V(P_{ik})$, and the remaining three vertices are defined similarly.
Similarly, let $v_{ik}, v_{i\ell}, v_{jk}, v_{j\ell}$ be the corresponding endpoints of $\mathcal{P}'$ in $Y$.
See Figure~\ref{fig1} for an illustration.
Let $<_X$ and $<_Y$ denote the cyclic ordering of $X$ and $Y$, respectively. 
Each segment $X_i$ (resp. $Y_i$) inherits the ordering $<_X$ (resp. $<_Y$).

\begin{define}\label{def:C4 types}
A 4-cycle $x_i y_k x_j y_\ell$ in $F$ is said to be of {\bf type $(\alpha,\beta)$}, where $\alpha, \beta \in \{0,1\}$ are defined as follows:
    \begin{itemize}
        \item[(1)] If both $u_{ik}<_{X}u_{i\ell}$ and $u_{jk}<_{X}u_{j\ell}$ hold, or both $u_{ik}>_{X}u_{i\ell}$ and $u_{jk}>_{X}u_{j\ell}$ hold, then let $\alpha=0$; otherwise, $\alpha=1$.
        \item[(2)] If both $v_{ik}<_{Y}v_{jk}$ and $v_{i\ell}<_{Y}v_{j\ell}$ hold, or both $v_{ik}>_{Y}v_{jk}$ and $v_{i\ell}>_{Y}v_{j\ell}$ hold, then let $\beta=0$; otherwise, $\beta=1$.
    \end{itemize}
\end{define}

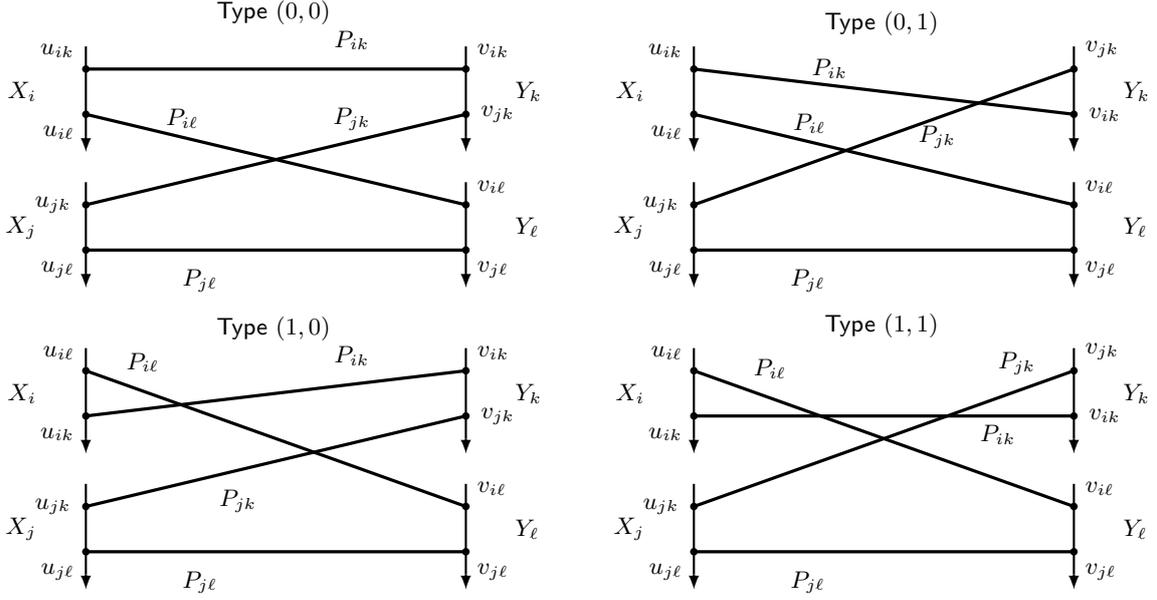
\begin{figure}[htb]
\captionsetup{justification=centering}
    \centering
\begin{tikzpicture}[
    every node/.style={font=\small\sf, inner sep=1pt},
    path/.style={thick, line width=1.2pt},
    hlpath/.style={ultra thick, line cap=round},
    xlabel/.style={text=black!80},
    ylabel/.style={text=black!80},
    >=latex,
    doublearr/.style={->, thick, line width=2pt, double, double distance=2pt}
]
\begin{scope}[local bounding box=1]

\draw[thick] (0,1.8) -- (0,3) node[pos=0.5, left=6mm] {$X_i$};
\draw[thick] (0,0) -- (0,1.2) node[pos=0.5, left=6mm] {$X_j$};

\draw[thick] (5,1.8) -- (5,3) node[pos=0.5, right=6mm] {$Y_k$};
\draw[thick] (5,0) -- (5,1.2) node[pos=0.5, right=6mm] {$Y_\ell$};
\draw[->, thick] (0,1.8) -- ++(0,-0.2);
\draw[->, thick] (0,0) -- ++(0,-0.2);
\draw[->, thick] (5,1.8) -- ++(0,-0.2);
\draw[->, thick] (5,0) -- ++(0,-0.2);

\node[circle,fill=black, label={[label distance=1mm]above left:$u_{ik}$}] (uik) at (0,2.7) {};
\node[circle,fill=black, label={[label distance=1mm]below left:$u_{i\ell}$}] (uil) at (0,2.1) {};
\node[circle,fill=black, label={[label distance=1mm]left:$u_{jk}$}] (ujk) at (0,0.9) {};
\node[circle,fill=black, label={[label distance=1mm]below left:$u_{j\ell}$}] (ujl) at (0,0.3) {};

\node[circle,fill=black, label={[label distance=1mm]above right:$v_{ik}$}] (vik) at (5,2.7) {};
\node[circle,fill=black, label={[label distance=1mm]right:$v_{jk}$}] (vjk) at (5,2.1) {};
\node[circle,fill=black, label={[label distance=1mm]above right:$v_{i\ell}$}] (vil) at (5,0.9) {};
\node[circle,fill=black, label={[label distance=1mm]below right:$v_{j\ell}$}] (vjl) at (5,0.3) {};

\draw[path] (0,2.7) -- (5,2.7) node[pos=0.7, above=2mm] {$P_{ik}$};
\draw[path] (0,2.1) -- (5,0.9) node[pos=0.2, above right] {$P_{i\ell}$};
\draw[path] (0,0.9) -- (5,2.1) node[pos=0.7, above=1mm] {$P_{jk}$};
\draw[path] (0,0.3) -- (5,0.3) node[pos=0.3, below=2mm] {$P_{j\ell}$};
\end{scope}

\begin{scope}[local bounding box=2, xshift=8cm]
\draw[thick] (0,1.8) -- (0,3) node[pos=0.5, left=6mm] {$X_i$};
\draw[thick] (0,0) -- (0,1.2) node[pos=0.5, left=6mm] {$X_j$};

\node[circle,fill=black, label={[label distance=1mm]above left:$u_{ik}$}] at (0,2.7) {};
\node[circle,fill=black, label={[label distance=1mm]below left:$u_{i\ell}$}] at (0,2.1) {};
\node[circle,fill=black, label={[label distance=1mm]left:$u_{jk}$}] at (0,0.9) {};
\node[circle,fill=black, label={[label distance=1mm]below left:$u_{j\ell}$}] at (0,0.3) {};

\draw[thick] (5,1.8) -- (5,3) node[pos=0.5, right=6mm] {$Y_k$};
\draw[thick] (5,0) -- (5,1.2) node[pos=0.5, right=6mm] {$Y_\ell$};

\node[circle,fill=black, label={[label distance=1mm]above right:$v_{jk}$}] at (5,2.7) {}; 
\node[circle,fill=black, label={[label distance=1mm]right:$v_{ik}$}] at (5,2.1) {};    
\node[circle,fill=black, label={[label distance=1mm]above right:$v_{i\ell}$}] at (5,0.9) {};
\node[circle,fill=black, label={[label distance=1mm]below right:$v_{j\ell}$}] at (5,0.3) {};

\draw[->, thick] (0,1.8) -- ++(0,-0.2);
\draw[->, thick] (0,0) -- ++(0,-0.2);
\draw[->, thick] (5,1.8) -- ++(0,-0.2);
\draw[->, thick] (5,0) -- ++(0,-0.2);

\draw[path] (0,2.7) -- (5,2.1) node[pos=0.3, above right] {$P_{ik}$}; 
\draw[path] (0,2.1) -- (5,0.9) node[pos=0.25, above right] {$P_{i\ell}$}; 
\draw[path] (0,0.9) -- (5,2.7) node[pos=0.5, right=4mm] {$P_{jk}$};  
\draw[path] (0,0.3) -- (5,0.3) node[pos=0.3, below=2mm] {$P_{j\ell}$}; 
\end{scope}

\begin{scope}[local bounding box=3, yshift=-4cm]
\draw[thick] (0,1.8) -- (0,3) node[pos=0.5, left=6mm] {$X_i$};
\draw[thick] (0,0) -- (0,1.2) node[pos=0.5, left=6mm] {$X_j$};

\node[circle,fill=black, label={[label distance=1mm]above left:$u_{i\ell}$}] at (0,2.7) {};
\node[circle,fill=black, label={[label distance=1mm]below left:$u_{ik}$}] at (0,2.1) {};
\node[circle,fill=black, label={[label distance=1mm]left:$u_{jk}$}] at (0,0.9) {};
\node[circle,fill=black, label={[label distance=1mm]below left:$u_{j\ell}$}] at (0,0.3) {};

\draw[thick] (5,1.8) -- (5,3) node[pos=0.5, right=6mm] {$Y_k$};
\draw[thick] (5,0) -- (5,1.2) node[pos=0.5, right=6mm] {$Y_\ell$};

\node[circle,fill=black, label={[label distance=1mm]above right:$v_{ik}$}] at (5,2.7) {};
\node[circle,fill=black, label={[label distance=1mm]right:$v_{jk}$}] at (5,2.1) {};
\node[circle,fill=black, label={[label distance=1mm]above right:$v_{i\ell}$}] at (5,0.9) {};
\node[circle,fill=black, label={[label distance=1mm]below right:$v_{j\ell}$}] at (5,0.3) {};
\draw[->, thick] (0,1.8) -- ++(0,-0.2);
\draw[->, thick] (0,0) -- ++(0,-0.2);
\draw[->, thick] (5,1.8) -- ++(0,-0.2);
\draw[->, thick] (5,0) -- ++(0,-0.2);

\draw[path] (0,2.1) -- (5,2.7) node[pos=0.7, above=2mm] {$P_{ik}$};
\draw[path] (0,2.7) -- (5,0.9) node[pos=0.15, above=2mm] {$P_{i\ell}$};
\draw[path] (0,0.9) -- (5,2.1) node[pos=0.4, below=2mm] {$P_{jk}$};
\draw[path] (0,0.3) -- (5,0.3) node[pos=0.3, below=2mm] {$P_{j\ell}$};
\end{scope}

\begin{scope}[local bounding box=4, xshift=8cm, yshift=-4cm]
\draw[thick] (0,1.8) -- (0,3) node[pos=0.5, left=6mm] {$X_i$};
\draw[thick] (0,0) -- (0,1.2) node[pos=0.5, left=6mm] {$X_j$};
\draw[->, thick] (0,1.8) -- ++(0,-0.2);
\draw[->, thick] (0,0) -- ++(0,-0.2);
\draw[->, thick] (5,1.8) -- ++(0,-0.2);
\draw[->, thick] (5,0) -- ++(0,-0.2);
\node[circle,fill=black, label={[label distance=1mm]above left:$u_{i\ell}$}] at (0,2.7) {};
\node[circle,fill=black, label={[label distance=1mm]below left:$u_{ik}$}] at (0,2.1) {};
\node[circle,fill=black, label={[label distance=1mm]left:$u_{jk}$}] at (0,0.9) {};
\node[circle,fill=black, label={[label distance=1mm]below left:$u_{j\ell}$}] at (0,0.3) {};

\draw[thick] (5,1.8) -- (5,3) node[pos=0.5, right=6mm] {$Y_k$};
\draw[thick] (5,0) -- (5,1.2) node[pos=0.5, right=6mm] {$Y_\ell$};

\node[circle,fill=black, label={[label distance=1mm]above right:$v_{jk}$}] at (5,2.7) {};
\node[circle,fill=black, label={[label distance=1mm]right:$v_{ik}$}] at (5,2.1) {};      
\node[circle,fill=black, label={[label distance=1mm]above right:$v_{i\ell}$}] at (5,0.9) {};
\node[circle,fill=black, label={[label distance=1mm]below right:$v_{j\ell}$}] at (5,0.3) {};

\draw[path] (0,2.1) -- (5,2.1) node[pos=0.85, above=5mm] {$P_{jk}$};
\draw[path] (0,2.7) -- (5,0.9) node[pos=0.2, above=2mm] {$P_{i\ell}$};
\draw[path] (0,0.9) -- (5,2.7) node[pos=0.8, below=3mm] {$P_{ik}$};  
\draw[path] (0,0.3) -- (5,0.3) node[pos=0.3, below=2mm] {$P_{j\ell}$}; 
\end{scope}

\node[above] at (1.north) {Type $(0,0)$};
\node[above] at (2.north) {Type $(0,1)$};
\node[above] at (3.north) {Type $(1,0)$};
\node[above] at (4.north) {Type $(1,1)$};
\end{tikzpicture}
\caption{Corresponding paths in $G$ for all types of 4-cycles in $F$.} \label{fig1}
\end{figure}

Figure~\ref{fig1} illustrates the four possible types of 4-cycles in $F$.  
Intuitively, \(\alpha = 0\) if and only if the order of \(u_{ik}, u_{i\ell}\) on \(X_i\) coincides with the order of \(u_{jk}, u_{j\ell}\) on \(X_j\) with respect to \(<_X\); similarly for \(\beta\).  
We note that the type of a 4-cycle is invariant under the choice of the cyclic ordering \(<_X\) on $X$.  
Indeed, reversing the ordering \(<_X\) on $X\) simultaneously reverses the relative orderings of both \(\{u_{ik}, u_{i\ell}\}\) and \(\{u_{jk}, u_{j\ell}\}\).  
Therefore, reversing \(<_X\) (and likewise \(<_Y\)) does not affect the type of the 4-cycle.
Below we examine the existence of 4-cycles of each type.  

Our first lemma shows that, in our setting, $F$ does not contain any 4-cycle of type $(0,0)$.  
An equivalent statement was proved in \cite{chen1998intersections},  
where it played a central role in showing that $F$ is $K_{3,257}$-free with the aid of the Erd\H{o}s--Szekeres theorem.

\begin{lemma}\label{lem:type 00}
    The graph $F$ contains no 4-cycle of type $(0,0)$.
\end{lemma}

\begin{proof}
    Suppose for a contradiction that $x_i y_k x_j y_\ell$ is a 4-cycle of type $(0,0)$.
    Keeping the previous notation for paths and endpoints, we can assume without loss of generality that
    \begin{align*}
        u_{ik} <_X u_{i\ell},\quad u_{jk}<_X u_{j\ell}, \quad v_{ik} <_Y v_{jk},\quad \mbox{and} \quad v_{i\ell} &<_Y v_{j\ell}.
    \end{align*}
    
For any distinct vertices $u,v\in V(X)$, let $\overrightarrow{X}[u,v]$ denote the path from $u$ to $v$ along the fixed cyclic ordering \(<_X\) of $X$.
The notation $\overrightarrow{Y}[u, v]$ is defined analogously for $Y$.
Construct two cycles $Q_1$ and $Q_2$ in $G$ as follows (Figure~\ref{fig:type00} depicts $Q_1$ in red and $Q_2$ in blue):
\begin{align*}
    Q_1 &= P_{ik}[u_{ik},v_{ik}] \cup \overrightarrow{Y}[v_{j\ell},v_{ik}] \cup P_{j\ell}[v_{j\ell},u_{j\ell}] \cup X_j[u_{j\ell},u_{jk}] \\
        &\quad \cup P_{jk}[u_{jk},v_{jk}] \cup \overrightarrow{Y}[v_{jk},v_{i\ell}] \cup P_{i\ell}[v_{i\ell},u_{i\ell}] \cup X_i[u_{i\ell},u_{ik}], \\
    Q_2 &= P_{ik}[u_{ik},v_{ik}] \cup Y_k[v_{ik},v_{jk}] \cup P_{jk}[v_{jk},u_{jk}] \cup \overrightarrow{X}[u_{i\ell},u_{jk}] \\
        &\quad \cup P_{i\ell}[u_{i\ell},v_{i\ell}] \cup Y_\ell[v_{i\ell},v_{j\ell}] \cup P_{j\ell}[v_{j\ell},u_{j\ell}] \cup \overrightarrow{X}[u_{j\ell},u_{ik}].
\end{align*}
It is straightforward to verify that \((Q_1,Q_2)\) forms a winning certificate, since 

\[
E(Q_1)\cup E(Q_2)\supseteq E(X)\cup E(Y) \quad \mbox{and} \quad
|Q_1| + |Q_2| = |X| + |Y| + 2\sum_{P\in \mathcal{P}'} |P| > |X| + |Y|,
\]
contradicting the maximality of \(X\) and \(Y\).
\end{proof}

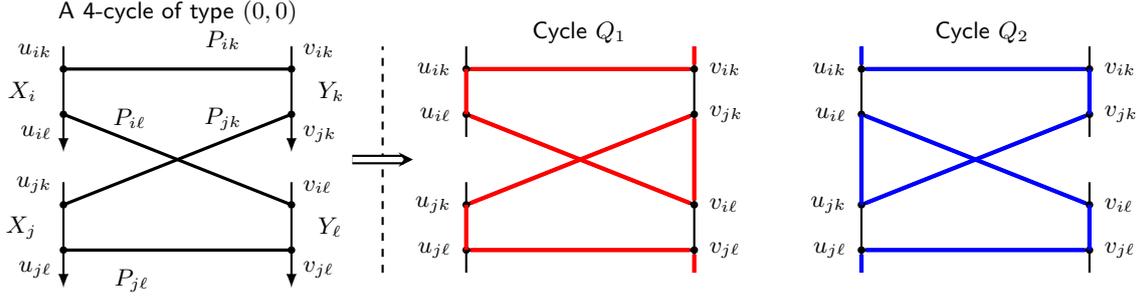
\begin{figure}[htb]
\captionsetup{justification=centering}
    \centering
\begin{tikzpicture}[
    every node/.style={font=\small\sf, inner sep=1pt},
    path/.style={thick, line width=1.2pt},
    hlpath/.style={ultra thick, line cap=round},
    xlabel/.style={text=black!80},
    ylabel/.style={text=black!80},
    >=latex,
    doublearr/.style={->, thick, line width=2pt, double, double distance=2pt}
]
\begin{scope}[local bounding box=1]
\draw[thick] (0,1.8) -- (0,3) node[pos=0.5, left=3mm] {$X_i$};
\draw[thick] (0,0) -- (0,1.2) node[pos=0.5, left=3mm] {$X_j$};

\draw[thick] (3,1.8) -- (3,3) node[pos=0.5, right=3mm] {$Y_k$};
\draw[thick] (3,0) -- (3,1.2) node[pos=0.5, right=3mm] {$Y_\ell$};

\draw[->, thick] (0,1.8) -- ++(0,-0.2);
\draw[->, thick] (0,0) -- ++(0,-0.2);
\draw[->, thick] (3,1.8) -- ++(0,-0.2);
\draw[->, thick] (3,0) -- ++(0,-0.2);

\node[circle,fill=black, label={[label distance=1mm]above left:$u_{ik}$}] (uik) at (0,2.7) {};
\node[circle,fill=black, label={[label distance=1mm]below left:$u_{i\ell}$}] (uil) at (0,2.1) {};
\node[circle,fill=black, label={[label distance=1mm]above left:$u_{jk}$}] (ujk) at (0,0.9) {};
\node[circle,fill=black, label={[label distance=1mm]below left:$u_{j\ell}$}] (ujl) at (0,0.3) {};

\node[circle,fill=black, label={[label distance=1mm]above right:$v_{ik}$}] (vik) at (3,2.7) {};
\node[circle,fill=black, label={[label distance=1mm]below right:$v_{jk}$}] (vjk) at (3,2.1) {};
\node[circle,fill=black, label={[label distance=1mm]above right:$v_{i\ell}$}] (vil) at (3,0.9) {};
\node[circle,fill=black, label={[label distance=1mm]below right:$v_{j\ell}$}] (vjl) at (3,0.3) {};

\draw[path] (0,2.7) -- (3,2.7) node[pos=0.7, above=2mm] {$P_{ik}$};
\draw[path] (0,2.1) -- (3,0.9) node[pos=0.2, above right] {$P_{i\ell}$};
\draw[path] (0,0.9) -- (3,2.1) node[pos=0.7, above=1mm] {$P_{jk}$};
\draw[path] (0,0.3) -- (3,0.3) node[pos=0.3, below=2mm] {$P_{j\ell}$};

\draw[dashed, thick] (4.2,0) -- (4.2,3); 
\draw[double, double distance=2pt, -stealth, line width=0.8pt] (3.8,1.5) -- (4.6,1.5);
\end{scope}

\begin{scope}[local bounding box=2, xshift=5.3cm]
\draw[thick] (0,1.8) -- (0,3) node[pos=0.5, left=3mm] {};
\draw[thick] (0,0) -- (0,1.2) node[pos=0.5, left=3mm] {};

\draw[thick] (3,1.8) -- (3,3) node[pos=0.5, right=3mm] {};
\draw[thick] (3,0) -- (3,1.2) node[pos=0.5, right=3mm] {};

\node[circle,fill=black, label={[label distance=1mm]left:$u_{ik}$}] (uik) at (0,2.7) {};
\node[circle,fill=black, label={[label distance=1mm]left:$u_{i\ell}$}] (uil) at (0,2.1) {};
\node[circle,fill=black, label={[label distance=1mm]left:$u_{jk}$}] (ujk) at (0,0.9) {};
\node[circle,fill=black, label={[label distance=1mm]left:$u_{j\ell}$}] (ujl) at (0,0.3) {};

\node[circle,fill=black, label={[label distance=1mm]right:$v_{ik}$}] (vik) at (3,2.7) {};
\node[circle,fill=black, label={[label distance=1mm]right:$v_{jk}$}] (vjk) at (3,2.1) {};
\node[circle,fill=black, label={[label distance=1mm]right:$v_{i\ell}$}] (vil) at (3,0.9) {};
\node[circle,fill=black, label={[label distance=1mm]right:$v_{j\ell}$}] (vjl) at (3,0.3) {};

\draw[path] (0,2.7) -- (3,2.7) node[pos=0.7, above=2mm] {};
\draw[path] (0,2.1) -- (3,0.9) node[pos=0.2, above right] {};
\draw[path] (0,0.9) -- (3,2.1) node[pos=0.7, above=1mm] {};
\draw[path] (0,0.3) -- (3,0.3) node[pos=0.3, below=2mm] {};

\end{scope}

\begin{scope}[red, ultra thick,opacity = 1]

\draw (uik) -- (vik); 
\draw (vik) -- (8.3,3); 
\draw (vjl) -- (8.3,0);
\draw (vjl) -- (ujl); 
\draw (ujl.center) -- (ujk.center); 
\draw (ujk) -- (vjk);
\draw (vil) -- (uil); 
\draw (uil.center) -- (uik.center); 
\draw (vjk) -- (vil);
\end{scope}

\begin{scope}[local bounding box=3, xshift=10.5cm]
\draw[thick] (0,1.8) -- (0,3) node[pos=0.5, left=3mm] {};
\draw[thick] (0,0) -- (0,1.2) node[pos=0.5, left=3mm] {};

\draw[thick] (3,1.8) -- (3,3) node[pos=0.5, right=3mm] {};
\draw[thick] (3,0) -- (3,1.2) node[pos=0.5, right=3mm] {};

\node[circle,fill=black, label={[label distance=1mm]left:$u_{ik}$}] (uik) at (0,2.7) {};
\node[circle,fill=black, label={[label distance=1mm]left:$u_{i\ell}$}] (uil) at (0,2.1) {};
\node[circle,fill=black, label={[label distance=1mm]left:$u_{jk}$}] (ujk) at (0,0.9) {};
\node[circle,fill=black, label={[label distance=1mm]left:$u_{j\ell}$}] (ujl) at (0,0.3) {};

\node[circle,fill=black, label={[label distance=1mm]right:$v_{ik}$}] (vik) at (3,2.7) {};
\node[circle,fill=black, label={[label distance=1mm]right:$v_{jk}$}] (vjk) at (3,2.1) {};
\node[circle,fill=black, label={[label distance=1mm]right:$v_{i\ell}$}] (vil) at (3,0.9) {};
\node[circle,fill=black, label={[label distance=1mm]right:$v_{j\ell}$}] (vjl) at (3,0.3) {};

\draw[path] (0,2.7) -- (3,2.7) node[pos=0.7, above=2mm] {};
\draw[path] (0,2.1) -- (3,0.9) node[pos=0.2, above right] {};
\draw[path] (0,0.9) -- (3,2.1) node[pos=0.7, above=1mm] {};
\draw[path] (0,0.3) -- (3,0.3) node[pos=0.3, below=2mm] {};

\end{scope}

\begin{scope}[blue,ultra thick,opacity = 1]

\draw (uik) -- (vik); 
\draw (vik.center) -- (vjk.center);
\draw (vjk) -- (ujk);

\draw (uil) -- (vil); 
\draw (vil.center) -- (vjl.center); 
\draw (vjl) -- (ujl); 
\draw (ujl) -- (10.5,0); 
\draw (uik) -- (10.5,3);
\draw (uil.center) -- (ujk.center);
\end{scope}

\node[above, xshift = -4mm] at (1.north) {A 4-cycle of type $(0,0)$};
\node[above] at (2.north) {Cycle $Q_1$};
\node[above,xshift = 1mm] at (3.north) {Cycle $Q_2$};

\end{tikzpicture}
\caption{A winning certificate $(Q_1,Q_2)$ corresponding to a 4-cycle of type $(0,0)$.} \label{fig:type00}
\end{figure}

Next, we consider two 4-cycles in $F$ of type $(1,0)$,  
which correspond to eight paths from $\mathcal{P}$.  
To streamline our arguments, we introduce the following notation.

\begin{define}
    For indices $1\leq i_1< j_1\leq m$ and $1\leq i_2< j_2\leq m$, the pairs $(i_1,j_1)$ and $(i_2,j_2)$ are {\bf crossing}, if $i_1,j_1,i_2,j_2$ are pairwise distinct and the interval $[i_1,j_1]$ contains exactly one element in $\{i_2,j_2\}$; otherwise, we say that $(i_1,j_1)$ and $(i_2,j_2)$ are {\bf non-crossing}.
\end{define}

We first show that a particular arrangement of the endpoints of these eight paths on $X$ and $Y$ guarantees the existence of a winning certificate.

\begin{figure}[htbp] 
\centering
\begin{subfigure}{\textwidth}
\centering
    \begin{tikzpicture}[
    every node/.style={font=\small\sf, inner sep=1pt},
    path/.style={thick, line width=1.2pt},
    hlpath/.style={ultra thick, line cap=round},
    xlabel/.style={text=black!80},
    ylabel/.style={text=black!80},
    >=latex,
    doublearr/.style={->, thick, line width=2pt, double, double distance=2pt},scale=0.85
    ]

\begin{scope}[local bounding box=1,xshift=-8mm]
\draw [fill={rgb,255:red,200;green,200;blue,200}, fill opacity=0.7, draw=none] 
    (0,10) ellipse [x radius=0.4, y radius=0.8];
\draw [fill={rgb,255:red,200;green,200;blue,200}, fill opacity=0.7, draw=none] 
    (0,4) ellipse [x radius=0.4, y radius=0.8];
\draw [fill={rgb,255:red,200;green,200;blue,200}, fill opacity=0.7, draw=none] 
    (2.7,10) ellipse [x radius=0.4, y radius=0.8];
\draw [fill={rgb,255:red,200;green,200;blue,200}, fill opacity=0.7, draw=none] 
    (2.7,7) ellipse [x radius=0.4, y radius=0.8];
\draw [fill={rgb,255:red,200;green,200;blue,200}, fill opacity=0.7, draw=none] 
    (-0.6,0.3) -- (0.6,0.3) -- (0,1.8) -- cycle;
\draw [fill={rgb,255:red,200;green,200;blue,200}, fill opacity=0.7, draw=none] 
    (-0.6,6.3) -- (0.6,6.3) -- (0,7.8) -- cycle;
\draw [fill={rgb,255:red,200;green,200;blue,200}, fill opacity=0.7, draw=none] 
    (2.3,4.8) rectangle (3.1,3.2);
\draw [fill={rgb,255:red,200;green,200;blue,200}, fill opacity=0.7, draw=none] 
    (2.3,1.8) rectangle (3.1,0.2);
\draw[thick] (0,9) -- (0,11) node[pos=0.0, left=3mm] {$X_{i_1}$};
\draw[thick] (0,6) -- (0,8) node[pos=0.0, left=3mm] {$X_{i_2}$};
\draw[thick] (0,3) -- (0,5) node[pos=0.0, left=3mm] {$X_{j_1}$};
\draw[thick] (0,0) -- (0,2) node[pos=0.0, left=3mm] {$X_{j_2}$};

\draw[thick] (3-0.3,9) -- (3-0.3,11) node[pos=0.0, right=3mm] {$Y_{k_1}$};
\draw[thick] (3-0.3,6) -- (3-0.3,8) node[pos=0.0, right=3mm] {$Y_{\ell_1}$};
\draw[thick] (3-0.3,3) -- (3-0.3,5) node[pos=0.0, right=3mm] {$Y_{k_2}$};
\draw[thick] (3-0.3,0) -- (3-0.3,2) node[pos=0.0, right=3mm] {$Y_{\ell_2}$};

\draw[->, thick] (0,9) -- ++(0,-0.2);
\draw[->, thick] (0,6) -- ++(0,-0.2);
\draw[->, thick] (3-0.3,3) -- ++(0,-0.2);
\draw[->, thick] (3-0.3,0) -- ++(0,-0.2);
\draw[->, thick] (0,3) -- ++(0,-0.2);
\draw[->, thick] (0,0) -- ++(0,-0.2);
\draw[->, thick] (3-0.3,9) -- ++(0,-0.2);
\draw[->, thick] (3-0.3,6) -- ++(0,-0.2);

\node[circle,fill=black, label={[label distance=1mm]left:$u_{i_1k_1}$}] (uik1) at (0,10.4) {};
\node[circle,fill=black, label={[label distance=1mm]left:$u_{i_1\ell_1}$}] (uil1) at (0,9.6) {};
\node[circle,fill=black, label={[label distance=1mm]left:$u_{j_1k_1}$}] (ujk1) at (0,3.6) {};
\node[circle,fill=black, label={[label distance=1mm]left:$u_{j_1\ell_1}$}] (ujl1) at (0,4.4) {};
\node[circle,fill=black, label={[label distance=1mm]left:$u_{i_2k_2}$}] (uik2) at (0,7.4) {};
\node[circle,fill=black, label={[label distance=1mm]left:$u_{i_2\ell_2}$}] (uil2) at (0,6.6) {};
\node[circle,fill=black, label={[label distance=1mm]left:$u_{j_2k_2}$}] (ujk2) at (0,0.6) {};
\node[circle,fill=black, label={[label distance=1mm]left:$u_{j_2\ell_2}$}] (ujl2) at (0,1.4) {};

\node[circle,fill=black, label={[label distance=1mm]right:$v_{i_1k_1}$}] (vik1) at (3-0.3,10.4) {};
\node[circle,fill=black, label={[label distance=1mm]right:$v_{j_1k_1}$}] (vjk1) at (3-0.3,9.6) {};
\node[circle,fill=black, label={[label distance=1mm]right:$v_{i_1\ell_1}$}] (vil1) at (3-0.3,7.4) {};
\node[circle,fill=black, label={[label distance=1mm]right:$v_{j_1\ell_1}$}] (vjl1) at (3-0.3,6.6) {};
\node[circle,fill=black, label={[label distance=1mm]right:$v_{i_2k_2}$}] (vik2) at (3-0.3,4.4) {};
\node[circle,fill=black, label={[label distance=1mm]right:$v_{j_2k_2}$}] (vjk2) at (3-0.3,3.6) {};
\node[circle,fill=black, label={[label distance=1mm]right:$v_{i_2\ell_2}$}] (vil2) at (3-0.3,1.4) {};
\node[circle,fill=black, label={[label distance=1mm]right:$v_{j_2\ell_2}$}] (vjl2) at (3-0.3,0.6) {};

\draw[path] (0,10.4) -- (3-0.3,10.4) node[pos=0.7, above=1mm] {$P_{i_1k_1}$};
\draw[path] (0,9.6) -- (3-0.3,7.4) node[pos=0.2, above=2mm] {$P_{i_1\ell_1}$};
\draw[path] (0,3.6) -- (3-0.3,9.6) node[pos=0.5, above=8mm] {$P_{j_1k_1}$};
\draw[path] (0,4.4) -- (3-0.3,6.6) node[pos=0.8, above=3mm] {$P_{j_1\ell_1}$};
\draw[path] (0,7.4) -- (3-0.3,4.4) node[pos=0.8, above=2.5mm] {$P_{i_2k_2}$};
\draw[path] (0,6.6) -- (3-0.3,1.4) node[pos=0.6, above=7mm] {$P_{i_2\ell_2}$};
\draw[path] (0,0.6) -- (3-0.3,3.6) node[pos=0.5, above=4mm] {$P_{j_2k_2}$};
\draw[path] (0,1.4) -- (3-0.3,0.6) node[pos=0.7, below=1.5mm] {$P_{j_2\ell_2}$};
\end{scope}

\draw[dashed, thick] (3.2,0) -- (3.2,11);
\draw[double, double distance=2pt, -stealth, line width=0.8pt] (2.8,5.5) -- (3.6,5.5);
\begin{scope}[local bounding box=2, xshift=6cm] 

\end{scope}

\begin{scope}[local bounding box=2, xshift=4.4cm]
\draw[thick] (0,9) -- (0,11) node[pos=0.5, left=6mm] {};
\draw[thick] (0,6) -- (0,8) node[pos=0.5, left=6mm] {};
\draw[thick] (0,3) -- (0,5) node[pos=0.5, left=6mm] {};
\draw[thick] (0,0) -- (0,2) node[pos=0.5, left=6mm] {};

\draw[thick] (3-0.3,9) -- (3-0.3,11) node[pos=0.5, right=6mm] {};
\draw[thick] (3-0.3,6) -- (3-0.3,8) node[pos=0.5, right=6mm] {};
\draw[thick] (3-0.3,3) -- (3-0.3,5) node[pos=0.5, right=6mm] {};
\draw[thick] (3-0.3,0) -- (3-0.3,2) node[pos=0.5, right=6mm] {};

\draw [fill={rgb,255:red,200;green,200;blue,200}, fill opacity=0.7, draw=none] 
    (-0.6,0.3) -- (0.6,0.3) -- (0,1.8) -- cycle;
\draw [fill={rgb,255:red,200;green,200;blue,200}, fill opacity=0.7, draw=none] 
    (-0.6,6.3) -- (0.6,6.3) -- (0,7.8) -- cycle;

\node[circle,fill=black, label={[label distance=1mm]left:$u_{i_1k_1}$}] (uik1) at (0,10.4) {};
\node[circle,fill=black, label={[label distance=1mm]left:$u_{i_1\ell_1}$}] (uil1) at (0,9.6) {};
\node[circle,fill=black, label={[label distance=1mm]left:$u_{j_1k_1}$}] (ujk1) at (0,3.6) {};
\node[circle,fill=black, label={[label distance=1mm]left:$u_{j_1\ell_1}$}] (ujl1) at (0,4.4) {};
\node[circle,fill=black, label={[label distance=1mm]left:$u_{i_2k_2}$}] (uik2) at (0,7.4) {};
\node[circle,fill=black, label={[label distance=1mm]left:$u_{i_2\ell_2}$}] (uil2) at (0,6.6) {};
\node[circle,fill=black, label={[label distance=1mm]left:$u_{j_2k_2}$}] (ujk2) at (0,0.6) {};
\node[circle,fill=black, label={[label distance=1mm]left:$u_{j_2\ell_2}$}] (ujl2) at (0,1.4) {};

\node[circle,fill=black, label={[label distance=0mm]right:$v_{i_1k_1}$}] (vik1) at (3-0.3,10.4) {};
\node[circle,fill=black, label={[label distance=0mm]right:$v_{j_1k_1}$}] (vjk1) at (3-0.3,9.6) {};
\node[circle,fill=black, label={[label distance=0mm]right:$v_{i_1\ell_1}$}] (vil1) at (3-0.3,7.4) {};
\node[circle,fill=black, label={[label distance=0mm]right:$v_{j_1\ell_1}$}] (vjl1) at (3-0.3,6.6) {};
\node[circle,fill=black, label={[label distance=0mm]right:$v_{i_2k_2}$}] (vik2) at (3-0.3,4.4) {};
\node[circle,fill=black, label={[label distance=0mm]right:$v_{j_2k_2}$}] (vjk2) at (3-0.3,3.6) {};
\node[circle,fill=black, label={[label distance=0mm]right:$v_{i_2\ell_2}$}] (vil2) at (3-0.3,1.4) {};
\node[circle,fill=black, label={[label distance=0mm]right:$v_{j_2\ell_2}$}] (vjl2) at (3-0.3,0.6) {};

\draw[path] (0,10.4) -- (3-0.3,10.4) node[pos=0.7, above=1mm]{};
\draw[path] (0,9.6) -- (3-0.3,7.4) node[pos=0.2, above=1mm]{};
\draw[path] (0,3.6) -- (3-0.3,9.6) node[pos=0.5, above=6mm]{};
\draw[path] (0,4.4) -- (3-0.3,6.6) node[pos=0.8, above=2mm]{};
\draw[path] (0,7.4) -- (3-0.3,4.4) node[pos=0.8, above=2mm]{};
\draw[path] (0,6.6) -- (3-0.3,1.4) node[pos=0.6, above=5mm]{};
\draw[path] (0,0.6) -- (3-0.3,3.6) node[pos=0.5, above=4mm]{};
\draw[path] (0,1.4) -- (3-0.3,0.6) node[pos=0.7, below=2mm]{};

\end{scope}

\begin{scope}[red, ultra thick, opacity = 1]

\draw (uik1) -- (vik1);
\draw (vik1) -- (7.1,11);
\draw (vjl2) -- (7.1,0);
\draw (vjl2) -- (ujl2); 
\draw (ujl2.center) -- (ujk2.center);
\draw (ujk2) -- (vjk2); 

\draw (vil2) -- (uil2); 
\draw (uil2.center) -- (uik2.center); 
\draw (uik2) -- (vik2); 

\draw (vjl1) -- (ujl1);
\draw (ujl1.center) -- (ujk1.center); 
\draw (ujk1) -- (vjk1); 
\draw (vil1) -- (uil1); 
\draw (uil1.center) -- (uik1.center); 
\draw (vjk1) -- (vil1);
\draw (vik2) -- (vjl1); 
\draw (vjk2) -- (vil2); 
\end{scope}

\begin{scope}[local bounding box=3, xshift=9.8cm]
\draw[thick] (0,9) -- (0,11) node[pos=0.5, left=6mm] {};
\draw[thick] (0,6) -- (0,8) node[pos=0.5, left=6mm] {};
\draw[thick] (0,3) -- (0,5) node[pos=0.5, left=6mm] {};
\draw[thick] (0,0) -- (0,2) node[pos=0.5, left=6mm] {};

\draw [fill={rgb,255:red,200;green,200;blue,200}, fill opacity=0.7, draw=none] 
    (-0.6,0.3) -- (0.6,0.3) -- (0,1.8) -- cycle;
\draw [fill={rgb,255:red,200;green,200;blue,200}, fill opacity=0.7, draw=none] 
    (-0.6,6.3) -- (0.6,6.3) -- (0,7.8) -- cycle;

\draw[thick] (3-0.3,9) -- (3-0.3,11) node[pos=0.5, right=6mm] {};
\draw[thick] (3-0.3,6) -- (3-0.3,8) node[pos=0.5, right=6mm] {};
\draw[thick] (3-0.3,3) -- (3-0.3,5) node[pos=0.5, right=6mm] {};
\draw[thick] (3-0.3,0) -- (3-0.3,2) node[pos=0.5, right=6mm] {};

\node[circle,fill=black, label={[label distance=0mm]left:$u_{i_1k_1}$}] (uik1) at (0,10.4) {};
\node[circle,fill=black, label={[label distance=0mm]left:$u_{i_1\ell_1}$}] (uil1) at (0,9.6) {};
\node[circle,fill=black, label={[label distance=0mm]left:$u_{j_1k_1}$}] (ujk1) at (0,3.6) {};
\node[circle,fill=black, label={[label distance=0mm]left:$u_{j_1\ell_1}$}] (ujl1) at (0,4.4) {};
\node[circle,fill=black, label={[label distance=0mm]left:$u_{i_2k_2}$}] (uik2) at (0,7.4) {};
\node[circle,fill=black, label={[label distance=0mm]left:$u_{i_2\ell_2}$}] (uil2) at (0,6.6) {};
\node[circle,fill=black, label={[label distance=0mm]left:$u_{j_2k_2}$}] (ujk2) at (0,0.6) {};
\node[circle,fill=black, label={[label distance=0mm]left:$u_{j_2\ell_2}$}] (ujl2) at (0,1.4) {};

\node[circle,fill=black, label={[label distance=1mm]right:$v_{i_1k_1}$}] (vik1) at (3-0.3,10.4) {};
\node[circle,fill=black, label={[label distance=1mm]right:$v_{j_1k_1}$}] (vjk1) at (3-0.3,9.6) {};
\node[circle,fill=black, label={[label distance=1mm]right:$v_{i_1\ell_1}$}] (vil1) at (3-0.3,7.4) {};
\node[circle,fill=black, label={[label distance=1mm]right:$v_{j_1\ell_1}$}] (vjl1) at (3-0.3,6.6) {};
\node[circle,fill=black, label={[label distance=1mm]right:$v_{i_2k_2}$}] (vik2) at (3-0.3,4.4) {};
\node[circle,fill=black, label={[label distance=1mm]right:$v_{j_2k_2}$}] (vjk2) at (3-0.3,3.6) {};
\node[circle,fill=black, label={[label distance=1mm]right:$v_{i_2\ell_2}$}] (vil2) at (3-0.3,1.4) {};
\node[circle,fill=black, label={[label distance=1mm]right:$v_{j_2\ell_2}$}] (vjl2) at (3-0.3,0.6) {};

\draw[path] (0,10.4) -- (3-0.3,10.4) node[pos=0.7, above=1mm]{};
\draw[path] (0,9.6) -- (3-0.3,7.4) node[pos=0.2, above=1mm]{};
\draw[path] (0,3.6) -- (3-0.3,9.6) node[pos=0.5, above=6mm]{};
\draw[path] (0,4.4) -- (3-0.3,6.6) node[pos=0.8, above=2mm]{};
\draw[path,dash pattern=on 2pt off 1pt] (0,7.4) -- (3-0.3,4.4) node[pos=0.8, above=2mm]{};
\draw[path,dash pattern=on 2pt off 1pt] (0,6.6) -- (3-0.3,1.4) node[pos=0.6, above=5mm]{};
\draw[path,dash pattern=on 2pt off 1pt] (0,0.6) -- (3-0.3,3.6) node[pos=0.5, above=4mm]{};
\draw[path,dash pattern=on 2pt off 1pt] (0,1.4) -- (3-0.3,0.6) node[pos=0.7, below=2mm]{};
\end{scope}

\begin{scope}[blue,ultra thick,opacity = 1]

\draw (uik1) -- (vik1); 
\draw (vik1.center) -- (vjk1.center);
\draw (vjk1) -- (ujk1); 

\draw (ujl2) -- (vjl2); 
\draw (vjl2.center) -- (vil2.center); 
\draw (vil2) -- (uil2); 

\draw (ujl1) -- (vjl1);
\draw (vjl1) -- (vil1); 
\draw (vil1.center) -- (uil1.center); 

\draw (uik2) -- (vik2);
\draw (vik2.center) -- (vjk2.center); 
\draw (vjk2) -- (ujk2);
\draw (ujk2) -- (9.8,0); 
\draw (uik1) -- (9.8,11);
\draw (uil1) -- (uik2); 
\draw (uil2) -- (ujl1); 
\draw (ujl2.center) -- (ujk1.center); 
\end{scope}

\node[above] at (1.north) {Two 4-cycles of type $(1,0)$};
\node[above] at (2.north) {Cycle $Q_1$};
\node[above] at (3.north) {Cycle $Q_2$};

\end{tikzpicture}
\caption{A winning certificate $(Q_1,Q_2)$ corresponding to a configuration of two 4-cycles of type $(1,0)$.} \label{fig:3a}
\end{subfigure}

\vspace{0.5cm}

\begin{subfigure}{\textwidth}
    \centering
    \begin{tikzpicture}[
    every node/.style={font=\small\sf, inner sep=1pt},
    path/.style={thick, line width=1.2pt},
    hlpath/.style={ultra thick, line cap=round},
    xlabel/.style={text=black!80},
    ylabel/.style={text=black!80},
    >=latex,
    doublearr/.style={->, thick, line width=2pt, double, double distance=2pt},scale=0.85
]

\begin{scope}[local bounding box=1,xshift=-8mm]
\draw [fill={rgb,255:red,200;green,200;blue,200}, fill opacity=0.7, draw=none] 
    (0,10) ellipse [x radius=0.4, y radius=0.8];
\draw [fill={rgb,255:red,200;green,200;blue,200}, fill opacity=0.7, draw=none] 
    (0,4) ellipse [x radius=0.4, y radius=0.8];
\draw [fill={rgb,255:red,200;green,200;blue,200}, fill opacity=0.7, draw=none] 
    (2.7,10) ellipse [x radius=0.4, y radius=0.8];
\draw [fill={rgb,255:red,200;green,200;blue,200}, fill opacity=0.7, draw=none] 
    (2.7,7) ellipse [x radius=0.4, y radius=0.8];
\draw [fill={rgb,255:red,200;green,200;blue,200}, fill opacity=0.7, draw=none] 
    (-0.6,1.7) -- (0.6,1.7) -- (0,0.2) -- cycle;
\draw [fill={rgb,255:red,200;green,200;blue,200}, fill opacity=0.7, draw=none] 
    (-0.6,7.7) -- (0.6,7.7) -- (0,6.2) -- cycle;
\draw [fill={rgb,255:red,200;green,200;blue,200}, fill opacity=0.7, draw=none] 
    (2.3,4.8) rectangle (3.1,3.2);
\draw [fill={rgb,255:red,200;green,200;blue,200}, fill opacity=0.7, draw=none] 
    (2.3,1.8) rectangle (3.1,0.2);
\draw[thick] (0,9) -- (0,11) node[pos=0.0, left=3mm] {$X_{i_1}$};
\draw[thick] (0,6) -- (0,8) node[pos=0.0, left=3mm] {$X_{i_2}$};
\draw[thick] (0,3) -- (0,5) node[pos=0.0, left=3mm] {$X_{j_1}$};
\draw[thick] (0,0) -- (0,2) node[pos=0.0, left=3mm] {$X_{j_2}$};

\draw[thick] (3-0.3,9) -- (3-0.3,11) node[pos=0.0, right=3mm] {$Y_{k_1}$};
\draw[thick] (3-0.3,6) -- (3-0.3,8) node[pos=0.0, right=3mm] {$Y_{\ell_1}$};
\draw[thick] (3-0.3,3) -- (3-0.3,5) node[pos=0.0, right=3mm] {$Y_{k_2}$};
\draw[thick] (3-0.3,0) -- (3-0.3,2) node[pos=0.0, right=3mm] {$Y_{\ell_2}$};

\draw[->, thick] (0,9) -- ++(0,-0.2);
\draw[->, thick] (0,6) -- ++(0,-0.2);
\draw[->, thick] (3-0.3,3) -- ++(0,-0.2);
\draw[->, thick] (3-0.3,0) -- ++(0,-0.2);
\draw[->, thick] (0,3) -- ++(0,-0.2);
\draw[->, thick] (0,0) -- ++(0,-0.2);
\draw[->, thick] (3-0.3,9) -- ++(0,-0.2);
\draw[->, thick] (3-0.3,6) -- ++(0,-0.2);

\node[circle,fill=black, label={[label distance=1mm]left:$u_{i_1k_1}$}] (uik1) at (0,10.4) {};
\node[circle,fill=black, label={[label distance=1mm]left:$u_{i_1\ell_1}$}] (uil1) at (0,9.6) {};
\node[circle,fill=black, label={[label distance=1mm]left:$u_{j_1\ell_1}$}] (ujl1) at (0,3.6) {};
\node[circle,fill=black, label={[label distance=1mm]left:$u_{j_1k_1}$}] (ujk1) at (0,4.4) {};
\node[circle,fill=black, label={[label distance=1mm]left:$u_{i_2\ell_2}$}] (uil2) at (0,7.4) {};
\node[circle,fill=black, label={[label distance=1mm]left:$u_{i_2k_2}$}] (uik2) at (0,6.6) {};
\node[circle,fill=black, label={[label distance=1mm]left:$u_{j_2k_2}$}] (ujk2) at (0,0.6) {};
\node[circle,fill=black, label={[label distance=1mm]left:$u_{j_2\ell_2}$}] (ujl2) at (0,1.4) {};

\node[circle,fill=black, label={[label distance=1mm]right:$v_{i_1k_1}$}] (vik1) at (3-0.3,10.4) {};
\node[circle,fill=black, label={[label distance=1mm]right:$v_{j_1k_1}$}] (vjk1) at (3-0.3,9.6) {};
\node[circle,fill=black, label={[label distance=1mm]right:$v_{i_1\ell_1}$}] (vil1) at (3-0.3,7.4) {};
\node[circle,fill=black, label={[label distance=1mm]right:$v_{j_1\ell_1}$}] (vjl1) at (3-0.3,6.6) {};
\node[circle,fill=black, label={[label distance=1mm]right:$v_{i_2k_2}$}] (vik2) at (3-0.3,4.4) {};
\node[circle,fill=black, label={[label distance=1mm]right:$v_{j_2k_2}$}] (vjk2) at (3-0.3,3.6) {};
\node[circle,fill=black, label={[label distance=1mm]right:$v_{i_2\ell_2}$}] (vil2) at (3-0.3,1.4) {};
\node[circle,fill=black, label={[label distance=1mm]right:$v_{j_2\ell_2}$}] (vjl2) at (3-0.3,0.6) {};

\draw[path] (0,10.4) -- (3-0.3,10.4) node[pos=0.7, above=1mm] {$P_{i_1k_1}$};
\draw[path] (0,9.6) -- (3-0.3,7.4) node[pos=0.2, above=2mm] {$P_{i_1\ell_1}$};
\draw[path] (0,3.6) -- (3-0.3,9.6) node[pos=0.5, above=8mm] {$P_{j_1k_1}$};
\draw[path] (0,4.4) -- (3-0.3,6.6) node[pos=0.8, above=3mm] {$P_{j_1\ell_1}$};
\draw[path] (0,6.6) -- (3-0.3,4.4) node[pos=0.8, above=2.5mm] {$P_{i_2k_2}$};
\draw[path] (0,7.4) -- (3-0.3,1.4) node[pos=0.6, above=7mm] {$P_{i_2\ell_2}$};
\draw[path] (0,1.4) -- (3-0.3,3.6) node[pos=0.5, above=4mm] {$P_{j_2k_2}$};
\draw[path] (0,0.6) -- (3-0.3,0.6) node[pos=0.7, below=1.5mm] {$P_{j_2\ell_2}$};
\end{scope}

\draw[dashed, thick] (3.2,0) -- (3.2,11); 
\draw[double, double distance=2pt, -stealth, line width=0.8pt] (2.8,5.5) -- (3.6,5.5);
\begin{scope}[local bounding box=2, xshift=6cm] 

\end{scope}

\begin{scope}[local bounding box=2, xshift=4.4cm]
\draw[thick] (0,9) -- (0,11) node[pos=0.5, left=6mm] {};
\draw[thick] (0,6) -- (0,8) node[pos=0.5, left=6mm] {};
\draw[thick] (0,3) -- (0,5) node[pos=0.5, left=6mm] {};
\draw[thick] (0,0) -- (0,2) node[pos=0.5, left=6mm] {};

\draw [fill={rgb,255:red,200;green,200;blue,200}, fill opacity=0.7, draw=none] 
    (-0.6,1.7) -- (0.6,1.7) -- (0,0.2) -- cycle;
\draw [fill={rgb,255:red,200;green,200;blue,200}, fill opacity=0.7, draw=none] 
    (-0.6,7.7) -- (0.6,7.7) -- (0,6.2) -- cycle;

\draw[thick] (3-0.3,9) -- (3-0.3,11) node[pos=0.5, right=6mm] {};
\draw[thick] (3-0.3,6) -- (3-0.3,8) node[pos=0.5, right=6mm] {};
\draw[thick] (3-0.3,3) -- (3-0.3,5) node[pos=0.5, right=6mm] {};
\draw[thick] (3-0.3,0) -- (3-0.3,2) node[pos=0.5, right=6mm] {};

\node[circle,fill=black, label={[label distance=1mm]left:$u_{i_1k_1}$}] (uik1) at (0,10.4) {};
\node[circle,fill=black, label={[label distance=1mm]left:$u_{i_1\ell_1}$}] (uil1) at (0,9.6) {};
\node[circle,fill=black, label={[label distance=1mm]left:$u_{j_1\ell_1}$}] (ujl1) at (0,3.6) {};
\node[circle,fill=black, label={[label distance=1mm]left:$u_{j_1k_1}$}] (ujk1) at (0,4.4) {};
\node[circle,fill=black, label={[label distance=1mm]left:$u_{i_2\ell_2}$}] (uil2) at (0,7.4) {};
\node[circle,fill=black, label={[label distance=1mm]left:$u_{i_2k_2}$}] (uik2) at (0,6.6) {};
\node[circle,fill=black, label={[label distance=1mm]left:$u_{j_2k_2}$}] (ujk2) at (0,0.6) {};
\node[circle,fill=black, label={[label distance=1mm]left:$u_{j_2\ell_2}$}] (ujl2) at (0,1.4) {};

\node[circle,fill=black, label={[label distance=0mm]right:$v_{i_1k_1}$}] (vik1) at (3-0.3,10.4) {};
\node[circle,fill=black, label={[label distance=0mm]right:$v_{j_1k_1}$}] (vjk1) at (3-0.3,9.6) {};
\node[circle,fill=black, label={[label distance=0mm]right:$v_{i_1\ell_1}$}] (vil1) at (3-0.3,7.4) {};
\node[circle,fill=black, label={[label distance=0mm]right:$v_{j_1\ell_1}$}] (vjl1) at (3-0.3,6.6) {};
\node[circle,fill=black, label={[label distance=0mm]right:$v_{i_2k_2}$}] (vik2) at (3-0.3,4.4) {};
\node[circle,fill=black, label={[label distance=0mm]right:$v_{j_2k_2}$}] (vjk2) at (3-0.3,3.6) {};
\node[circle,fill=black, label={[label distance=0mm]right:$v_{i_2\ell_2}$}] (vil2) at (3-0.3,1.4) {};
\node[circle,fill=black, label={[label distance=0mm]right:$v_{j_2\ell_2}$}] (vjl2) at (3-0.3,0.6) {};

\draw[path] (0,10.4) -- (3-0.3,10.4) node[pos=0.7, above=1mm]{};
\draw[path] (0,9.6) -- (3-0.3,7.4) node[pos=0.2, above=1mm]{};
\draw[path] (0,3.6) -- (3-0.3,9.6) node[pos=0.5, above=6mm]{};
\draw[path] (0,4.4) -- (3-0.3,6.6) node[pos=0.8, above=2mm]{};
\draw[path] (0,6.6) -- (3-0.3,4.4) node[pos=0.8, above=2mm]{};
\draw[path] (0,7.4) -- (3-0.3,1.4) node[pos=0.6, above=5mm]{};
\draw[path] (0,1.4) -- (3-0.3,3.6) node[pos=0.5, above=4mm]{};
\draw[path] (0,0.6) -- (3-0.3,0.6) node[pos=0.7, below=2mm]{};

\end{scope}

\begin{scope}[red, ultra thick, opacity = 1]
\draw (uik1) -- (vik1); 
\draw (vik1) -- (7.1,11); 
\draw (vjl2) -- (7.1,0);
\draw (vjl2) -- (ujk2); 
\draw (ujl2.center) -- (ujk2.center);
\draw (ujl2) -- (vjk2); 

\draw (vil2) -- (uil2); 
\draw (uil2.center) -- (uik2.center); 
\draw (uik2) -- (vik2); 

\draw (vjl1) -- (ujk1);
\draw (ujl1.center) -- (ujk1.center); 
\draw (ujl1) -- (vjk1); 

\draw (vil1) -- (uil1);
\draw (uil1.center) -- (uik1.center); 
\draw (vjk1) -- (vil1); 
\draw (vik2) -- (vjl1); 
\draw (vjk2) -- (vil2); 
\end{scope}

\begin{scope}[local bounding box=3, xshift=9.8cm]
\draw[thick] (0,9) -- (0,11) node[pos=0.5, left=6mm] {};
\draw[thick] (0,6) -- (0,8) node[pos=0.5, left=6mm] {};
\draw[thick] (0,3) -- (0,5) node[pos=0.5, left=6mm] {};
\draw[thick] (0,0) -- (0,2) node[pos=0.5, left=6mm] {};

\draw[thick] (3-0.3,9) -- (3-0.3,11) node[pos=0.5, right=6mm] {};
\draw[thick] (3-0.3,6) -- (3-0.3,8) node[pos=0.5, right=6mm] {};
\draw[thick] (3-0.3,3) -- (3-0.3,5) node[pos=0.5, right=6mm] {};
\draw[thick] (3-0.3,0) -- (3-0.3,2) node[pos=0.5, right=6mm] {};

\draw [fill={rgb,255:red,200;green,200;blue,200}, fill opacity=0.7, draw=none] 
    (-0.6,1.7) -- (0.6,1.7) -- (0,0.2) -- cycle;
\draw [fill={rgb,255:red,200;green,200;blue,200}, fill opacity=0.7, draw=none] 
    (-0.6,7.7) -- (0.6,7.7) -- (0,6.2) -- cycle;

\node[circle,fill=black, label={[label distance=0mm]left:$u_{i_1k_1}$}] (uik1) at (0,10.4) {};
\node[circle,fill=black, label={[label distance=0mm]left:$u_{i_1\ell_1}$}] (uil1) at (0,9.6) {};
\node[circle,fill=black, label={[label distance=1mm]left:$u_{j_1\ell_1}$}] (ujl1) at (0,3.6) {};
\node[circle,fill=black, label={[label distance=1mm]left:$u_{j_1k_1}$}] (ujk1) at (0,4.4) {};
\node[circle,fill=black, label={[label distance=1mm]left:$u_{i_2\ell_2}$}] (uil2) at (0,7.4) {};
\node[circle,fill=black, label={[label distance=1mm]left:$u_{i_2k_2}$}] (uik2) at (0,6.6) {};
\node[circle,fill=black, label={[label distance=0mm]left:$u_{j_2k_2}$}] (ujk2) at (0,0.6) {};
\node[circle,fill=black, label={[label distance=0mm]left:$u_{j_2\ell_2}$}] (ujl2) at (0,1.4) {};

\node[circle,fill=black, label={[label distance=1mm]right:$v_{i_1k_1}$}] (vik1) at (3-0.3,10.4) {};
\node[circle,fill=black, label={[label distance=1mm]right:$v_{j_1k_1}$}] (vjk1) at (3-0.3,9.6) {};
\node[circle,fill=black, label={[label distance=1mm]right:$v_{i_1\ell_1}$}] (vil1) at (3-0.3,7.4) {};
\node[circle,fill=black, label={[label distance=1mm]right:$v_{j_1\ell_1}$}] (vjl1) at (3-0.3,6.6) {};
\node[circle,fill=black, label={[label distance=1mm]right:$v_{i_2k_2}$}] (vik2) at (3-0.3,4.4) {};
\node[circle,fill=black, label={[label distance=1mm]right:$v_{j_2k_2}$}] (vjk2) at (3-0.3,3.6) {};
\node[circle,fill=black, label={[label distance=1mm]right:$v_{i_2\ell_2}$}] (vil2) at (3-0.3,1.4) {};
\node[circle,fill=black, label={[label distance=1mm]right:$v_{j_2\ell_2}$}] (vjl2) at (3-0.3,0.6) {};

\draw[path] (0,10.4) -- (3-0.3,10.4) node[pos=0.7, above=1mm]{};
\draw[path] (0,9.6) -- (3-0.3,7.4) node[pos=0.2, above=1mm]{};
\draw[path] (0,3.6) -- (3-0.3,9.6) node[pos=0.5, above=6mm]{};
\draw[path] (0,4.4) -- (3-0.3,6.6) node[pos=0.8, above=2mm]{};
\draw[path] (0,6.6) -- (3-0.3,4.4) node[pos=0.8, above=2mm]{};
\draw[path] (0,7.4) -- (3-0.3,1.4) node[pos=0.6, above=5mm]{};
\draw[path] (0,1.4) -- (3-0.3,3.6) node[pos=0.5, above=4mm]{};
\draw[path] (0,0.6) -- (3-0.3,0.6) node[pos=0.7, below=2mm]{};
\end{scope}

\begin{scope}[blue,ultra thick,opacity = 1]

\draw (uik1) -- (vik1); 
\draw (vik1.center) -- (vjk1.center); 
\draw (vjk1) -- (ujl1); 

\draw (ujl2) -- (vjk2); 
\draw (vjl2.center) -- (vil2.center); 
\draw (vik2) -- (uik2);

\draw (ujl1) -- (vjk1);
\draw (vjl1) -- (vil1);
\draw (vil1.center) -- (uil1.center); 

\draw (uik2) -- (vik2); 
\draw (vik2.center) -- (vjk2.center);
\draw (vjl2) -- (ujk2); 
\draw (ujk2) -- (9.8,0); 
\draw (uik1) -- (9.8,11);
\draw (uil1) -- (uil2); 
\draw (uil2) -- (vil2); 
\draw (ujl2.center) -- (ujl1.center); 
\draw (uik2) -- (ujk1);
\draw (vjl1) -- (ujk1);
\end{scope}

\node[above] at (1.north) {Two 4-cycles of type $(1,0)$};
\node[above] at (2.north) {Cycle $Q_1$};
\node[above] at (3.north) {Cycle $Q_2$};

\end{tikzpicture}
\caption{The counter configuration from (a) with reversed order of pairs in triangle-shaped shadows.} \label{fig:3b}
\end{subfigure}
\caption{Key steps in the proof of Lemma~\ref{lem:crossing int}.}
\label{fig:both}
\end{figure}
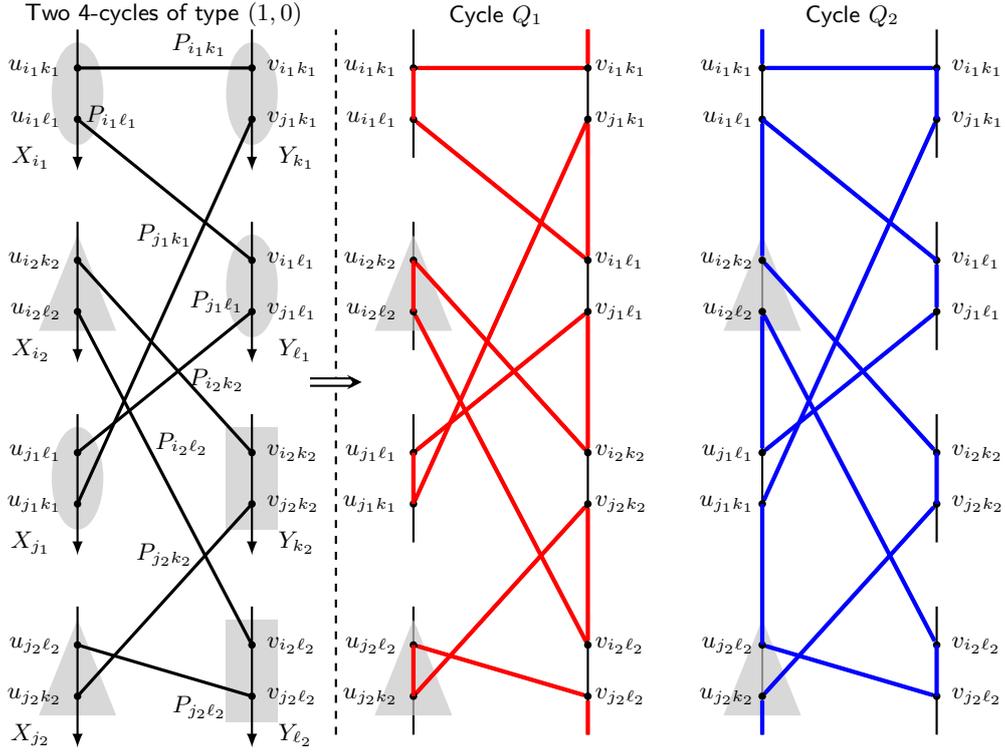
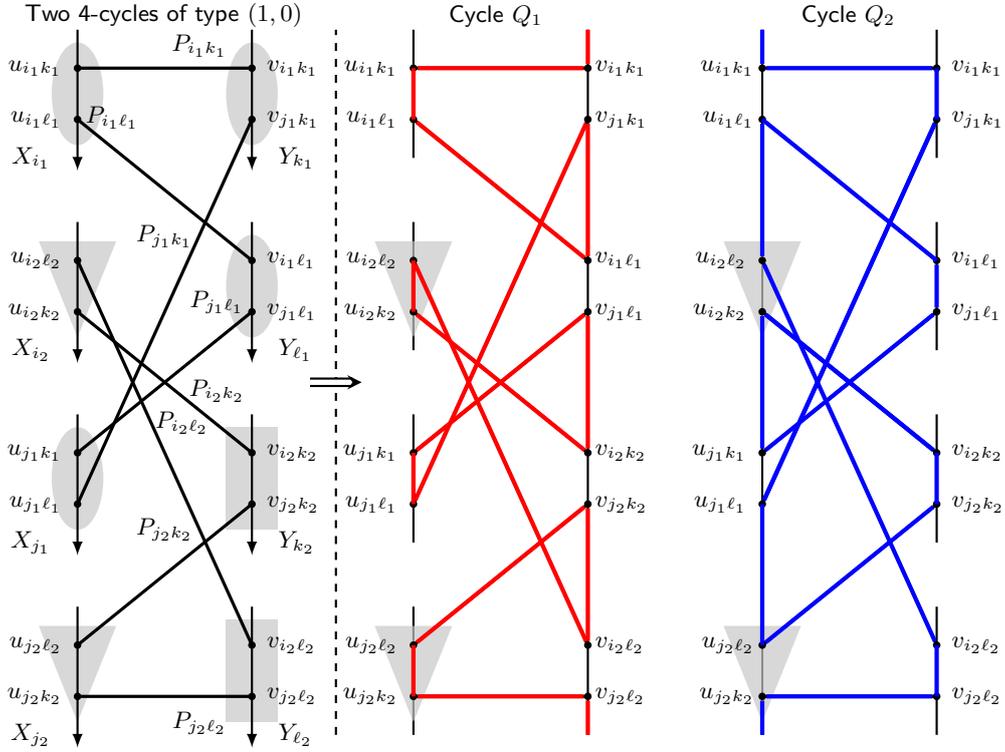

\begin{lemma}\label{lem:crossing int}
    Let $x_{i_1}y_{k_1}x_{j_1}y_{\ell_1}$ and $x_{i_2}y_{k_2}x_{j_2}y_{\ell_2}$ be two 4-cycles of type $(1,0)$ in $F$.
    If $(i_1,j_1)$ and $(i_2,j_2)$ are crossing, then either $\{k_1,\ell_1\}\cap\{k_2,\ell_2\}\neq \emptyset$, or $(k_1,\ell_1)$ and $(k_2,\ell_2)$ are crossing.
\end{lemma}

\begin{proof}
Let $(i_1,j_1)$ and $(i_2,j_2)$ be crossing.  
Suppose, for the sake of contradiction, that $\{k_1,\ell_1\}\cap\{k_2,\ell_2\} = \emptyset$ and that $(k_1,\ell_1)$ and $(k_2,\ell_2)$ are non-crossing.  
Let $\mathcal{P}'$ denote the collection of the eight disjoint paths in $\mathcal{P}$ corresponding to the edges in the two 4-cycles.

We may assume, without loss of generality, that the four segments $X_{i_1}, X_{i_2}, X_{j_1}, X_{j_2}$ are cyclically arranged on $X$, and that $Y_{k_1}, Y_{\ell_1}, Y_{k_2}, Y_{\ell_2}$ are cyclically arranged on $Y$; 
see Figure~\ref{fig:both}.

By considering all possible relative orderings of the endpoints of the paths in $\mathcal{P}'$,  
we claim that there are four non-isomorphic cases that need to be examined in general.
To see this, we first consider the four pairs of endpoints corresponding to the 4-cycle $x_{i_1}y_{k_1}x_{j_1}y_{\ell_1}$,  
which are highlighted by the ellipse-shaped shadows in Figure~\ref{fig:both}.
Since this 4-cycle is of type $(\alpha,\beta)=(0,1)$, we can select the initial cyclic orderings of $X$ and $Y$ such that $u_{i_1k_1}<_X u_{i_1\ell_1}$ and $v_{i_1k_1}<_Y v_{j_1k_1}$.
Then since $\alpha=0$ and $\beta=1$, it follows that $u_{j_1\ell_1}<_X u_{j_1k_1}$ and $v_{i_1\ell_1}<_Y v_{j_1\ell_1}$.
In this way, we could assume that the four pairs covered by the ellipse-shaped shadows are fixed.
It remains to consider the four pairs of endpoints corresponding to the second 4-cycle $x_{i_2}y_{k_2}x_{j_2}y_{\ell_2}$,
where the pairs $\{u_{i_2k_2}, u_{i_2\ell_2}\}$ and $\{u_{j_2\ell_2}, u_{j_2k_2}\}$ are highlighted by the triangle-shaped shadows
and the pairs $\{v_{i_2k_2}, v_{j_2k_2}\}$ and $\{v_{i_2\ell_2}, v_{j_2\ell_2}\}$ are highlighted by the rectangle-shaped shadows in Figure~\ref{fig:both}.
Since this cycle has type $(0,1)$,  
the relative orderings in the two pairs covered by the triangle-shaped shadows are consistent (as $\alpha=0$), yielding two possibilities (see the differences between Figures~\ref{fig:3a} and \ref{fig:3b}).  
Similarly, the relative orderings in the two pairs covered by the rectangle-shaped shadows are inconsistent (as $\beta=1$), also giving two possibilities.
This proves the claim.

Therefore, it suffices to discuss the four cases, under the alternative orderings of the pairs covered by the triangle-shaped and rectangle-shaped shadows.
Consider one such instance in Figure~\ref{fig:3a}.
In this case, there exist two cycles $Q_1$ and $Q_2$ (depicted by red and blue, respectively) satisfying
$$E(Q_1)\cup E(Q_2)\supseteq E(X)\cup E(Y) \quad \mbox{and} \quad |Q_1|+|Q_2|=|X|+|Y|+2\sum\limits_{P\in \mathcal{P}'}|P|>|X|+|Y|,$$ contradicting the maximality of $X$ and $Y$.
This forms a winning certificate $(Q_1,Q_2)$.
Figure~\ref{fig:3b} considers another instance, where we preserve the relative orderings of pairs covered by the rectangle-shaped shadows, and reverse that covered by the triangle-shaped shadows.
It is clear from the figure that the red and blue lines yield another winning certificate.
The verifications for the remaining two cases are similar, which we omit here.
\end{proof}

The following lemma provides a key tool for obtaining an upper bound on the number of copies of $K_{2,7}$ in $F$.
For $x \in V(F)$, let $N_F(x)$ denote its neighborhood.
For $1\leq i<j\leq m$, we define $a_{ij}=\big|N_F(x_i)\cap N_F(x_j)\big|$.

\begin{lemma}\label{lem:K29 error}
Let $L = \{(i,j) : a_{ij} \ge 7, ~ 1 \le i < j \le m \}$. 
Then any two pairs $(i_1,j_1)$ and $(i_2,j_2)$ in $L$ are non-crossing.
\end{lemma}

\begin{proof}
    Consider the pair $(i_1,j_1)$ and a copy of $K_{2,7}$ in $F$ with vertex set $\{x_{i_1},x_{j_1}\} \cup \{y_{k_t}:t\in[7]\}$.
    Let $\mathcal{Q}$ be the set of paths in $\mathcal{P}$ that correspond to the edges of this $K_{2,7}$.
    Each segment $Y_{k_t}$ contains two endpoints $v_{i_1k_t},v_{j_1k_t}$ from two paths in $\mathcal{Q}$ 
    whose relative ordering can be either $v_{i_1k_t} <_Y v_{j_1k_t}$ or $v_{i_1k_t} >_Y v_{j_1k_t}$.
    By the pigeonhole principle, at least $\lceil7/2\rceil=4$ distinct $Y$-segments share the same endpoint ordering.
    By Lemma~\ref{lem:type 00}, there is no 4-cycle of type $(0,0)$ in $F$.
    Therefore, we obtain a subgraph $K_{2,4}$ that only contains 4-cycles of type $(1,0)$.

    Applying above arguments to pairs $(i_1,j_1)$ and $(i_2,j_2)$, we obtain indices $1\leq k_1<\cdots<k_4\leq m$ and $1\leq \ell_1<\cdots<\ell_4\leq m$ such that both $F[\{x_{i_1},x_{j_1},y_{k_1},\cdots,y_{k_4}\}]$ and $F[\{x_{i_2},x_{j_2},y_{\ell_1},\cdots,y_{\ell_4}\}]$ are isomorphic to $K_{2,4}$, and any 4-cycle in these subgraphs has type $(1,0)$.
    
    We claim that there are two indices in $\{k_1,\cdots, k_4\}$ and two indices in $\{\ell_1,\cdots, \ell_4\}$ that are pairwise distinct and form non-crossing pairs.
    Suppose there are three indices $\{h_1,h_2,h_3\} \subseteq \{\ell_1,...,\ell_4\}$ distinct from $\{k_1,k_2\}$.
    By the pigeonhole principle, two of $h_1,h_2,h_3$ (say $h_1,h_2$) must both lie in the interval $(k_1,k_2)$ or both in $[1,k_1)\cup(k_2,m]$.
    Then the two disjoint pairs $(k_1,k_2)$ and $(h_1,h_2)$ are non-crossing, as claimed.
    Otherwise, we have $\{k_1,k_2\}\subseteq \{\ell_1,...,\ell_4\}$;
    similarly, we may assume that $\{k_3,k_4\}\subseteq \{\ell_1,...,\ell_4\}$. 
    This implies $k_t=\ell_t$ for $t\in [4]$.
    Thus, the two pairs $(k_1,k_2)$ and $(\ell_3,\ell_4)$  are non-crossing and share no common element, which proves the claim.
    
   Using this claim, we can deduce from Lemma~\ref{lem:crossing int} that $(i_1,j_1)$ and $(i_2,j_2)$ are non-crossing, completing the proof.
\end{proof}

We need one more lemma that provides an upper bound on the number of pairs that are pairwise non-crossing.

\begin{lemma}\label{lem:upp bound noncross}
Let $m \ge 2$ be an integer, and let $\{(i_t,j_t)\}_{t=1}^r$ be a collection of $r$ distinct pairs with $1 \le i_t < j_t \le m$ for all $t \in [r]$, such that any two pairs are non-crossing.  
Then $r \le 2m - 3$.
\end{lemma}

\begin{proof}
    Let $f(m)$ denote the maximum number of such pairs.  
    We proceed by induction on $m$ to show that $f(m)\leq 2m-3$.
    First observe that $f(2)=1$ and $f(m)$ is non-decreasing in $m$.
    Assume that $f(k)\leq 2k-3$ holds for all $2\leq k\leq m-1$.
    Consider a collection of pairs $\{(i_t,j_t)\}_{t=1}^{f(m)}$ of size $f(m)$.
    Define $\mathcal{I}=\{[i_t,j_t]:t\in [f(m)]\}\setminus\{[1,m]\}$ to be a collection of close intervals.
    An interval in $\mathcal{I}$ is called {\it maximal}, if it is maximal under interval inclusion.
    Without loss of generality, we can assume that $\mathcal{I}'=\{[i_t,j_t]\}_{t=1}^s$ contains all maximal intervals in $\mathcal{I}$.
    We point out that these maximal intervals must be pairwise internally-disjoint, because any two pairs $(i_t,j_t), (i_{t'},j_{t'})$ are non-crossing. 
    Moreover, each interval in $\mathcal{I}$ belongs to exactly one maximal interval. 

    If $|\mathcal{I}'|=1$, then each pair $(i_t,j_t)$ except for $(1,m)$ satisfies $i_1\leq i_t< j_t\leq j_1$.
    This implies $$f(m)\leq 1+f(j_1-i_1+1)\leq 1+f(m-1)\leq 2m-4.$$
    Now we may assume $s=|\mathcal{I}'|\geq 2$. 
    Then, using properties mentioned above, for each pair $(i_t,j_t)\in \mathcal{I}$ except for $(1,m)$, there is exactly one index $u\in [s]$ such that $i_u\leq i_t<j_t\leq j_u$.
    This yields $$f(m)\leq 1+\sum\limits_{u\in [s]}f(j_u-i_u+1)\leq 1+2\sum\limits_{u\in [s]}\big(j_u-i_u\big)-s\leq 1+2(m-1)-s\leq 2m-3,$$
where the second inequality follows by induction and the third inequality holds because $[i_u,j_u]$ for all $u\in [s]$ are pairwise internally-disjoint. The proof is complete. 
\end{proof}

We are ready to prove Theorem~\ref{thm:lct upp}.

\begin{proof}[\bf Proof of Theorem~\ref{thm:lct upp}.]
Let $X, Y$ be two longest cycles in a graph $G$ which shares $m$ vertices. 
Consider any collection $\mathcal{P}$ of disjoint $(X-M,Y-M)$-paths in $G$, where $M=V(X)\cap V(Y)$.
Let $F := F(X,Y,\mathcal{P})$ be defined from Definition~\ref{def:aux F}, and let $\mathcal{E}=e(F)=|\mathcal{P}|$.
We aim to show that $$|\mathcal{P}|=\mathcal{E}\leq \sqrt{10}m^\frac{3}{2}+\frac{m}{2}.$$
We may assume $\mathcal{E}\geq m$. Under this assumption, the convexity of the function $\frac{x(x-1)}{2}$ yields $$\sum\limits_{1\leq i<j\leq m}a_{ij}=\sum\limits_{i\in [m]}\dbinom{\deg_F(y_i)}{2}\geq m\cdot\dbinom{\sum\limits_{i\in [m]}\deg_F(y_i)\big/m}{2}=m\binom{\mathcal{E}/m}{2}=\frac{\mathcal{E}(\mathcal{E}-m)}{2m}.$$
Let $\mathbf{1}_A$ denote the indicator function of event $A$.
Then we have $$\sum\limits_{1\leq i<j\leq m}a_{ij}\cdot\mathbf{1}_{\{a_{ij}\geq 7\}}\geq\sum\limits_{1\leq i<j\leq m}a_{ij}-6\dbinom{m}{2}\geq \frac{\mathcal{E}(\mathcal{E}-m)}{2m}-3m(m-1).$$
Recall that $L = \{(i,j) : a_{ij} \ge 7, ~ 1 \le i < j \le m \}$. 
Evidently $a_{ij}=\big|N_F(x_i)\cap N_F(x_j)\big|\leq m$ for all $1\leq i<j\leq m$, so we can deduce that $$|L|=\sum\limits_{1\leq i<j\leq m}\mathbf{1}_{\{a_{ij}\geq7\}}\geq \sum\limits_{1\leq i<j\leq m}a_{ij}\cdot\mathbf{1}_{\{a_{ij}\geq7\}}\big/m\geq \frac{\mathcal{E}(\mathcal{E}-m)}{2m^2}-3(m-1).$$
By Lemma~\ref{lem:K29 error}, the pairs in $L$ are pairwise non-crossing.
Applying Lemma~\ref{lem:upp bound noncross}, we obtain $$\frac{\mathcal{E}(\mathcal{E}-m)}{2m^2}-3(m-1)\leq |L|\leq 2m-3.$$
Solving this inequality gives $\mathcal{E}^2-m \mathcal{E}-(10m^3-12m^2)\leq 0$,
which implies $|\mathcal{P}|=\mathcal{E}\leq \sqrt{10}m^\frac{3}{2}+\frac{m}{2},$ as desired. 
Thus, there are at most $\sqrt{10}\, m^{3/2} + \frac{m}{2}$ disjoint paths in $G$ between $X-M$ and $Y-M$.  
By Menger's Theorem, this implies the existence of a vertex cut of size at most $\sqrt{10}\, m^{3/2} + \frac{m}{2}$ separating $X-M$ and $Y-M$.  
Adding the set $M$ of size $m$, we obtain a vertex cut of size at most $\sqrt{10}\, m^{3/2} + \frac{3m}{2}$ separating $X$ and $Y$ in $G$, completing the proof.
\end{proof}

We conclude this section with a proof of  Corollary~\ref{cor:vert-tran-c(G)}.

\begin{proof}[\bf Proof of Corollary~\ref{cor:vert-tran-c(G)}.]
Let $G$ be a connected $d$-regular vertex-transitive graph on $n$ vertices, where $n>d\geq 2$.
Then $G$ is \(\Omega(d)\)-connected (see~\cite{godsil2013algebraic}).
By Corollary~\ref{cor:smith}, every two longest cycles in $G$ intersect in $\Omega(d^{2/3})$ vertices.
Let $A$ denote the set of vertices in a fixed longest cycle. 
Then $A$ is a $t$-transversal in $G$ for some $t=\Omega(d^{2/3})$.
Applying equation~\eqref{equ:DeVos}, we obtain $$c(G)\geq \sqrt{tn}=\sqrt{\Omega(d^{2/3})n}=\Omega(d^\frac{1}{3}\cdot n^\frac{1}{2}).$$
This provides the existence of a desired cycle in $G$. 
\end{proof}


\section{Proof of Theorem~\ref{thm:transitive babai}}
In this section we prove Theorem~\ref{thm:transitive babai}, which asserts that any two longest cycles in a connected vertex-transitive graph must share a sufficiently large set of vertices.

To proceed, we employ a structural lemma on connected vertex-transitive graphs due to DeVos and Mohar~\cite{devos2006small}. 
It asserts that in a connected vertex-transitive graph, a small separator cannot disconnect two large sets unless the graph possesses a special global structure. 
For a formal statement, we first introduce some notation. 
Let $G$ be a connected vertex-transitive graph. 
For any $x,y\in V(G)$, let $\mathrm{dist}(x,y)$ denote the length of a shortest $(x,y)$-path in $G$, and define
\[
\mathrm{diam}(G)=\max_{x,y\in V(G)} \mathrm{dist}(x,y).
\]
The \textit{neighborhood} $N(A)$ of a subset $A \subseteq V(G)$ denotes the set of all vertices in $V(G) \setminus A$ adjacent to some vertex in $A$.
    A partition $\sigma=\{B_1,\ldots, B_k\}$ of $V(G)$ is called a \textit{system of imprimitivity} if for every automorphism $g$ of $G$ and every set $B \in \sigma$, the image $B^g:=g(B)$ remains in $\sigma$. These vertex sets $B \in \sigma$ are called \textit{blocks}.
    A \textit{cyclic system} $\vec{\sigma}$ on $G$ is a system of imprimitivity $\sigma$ equipped with a cyclic ordering of the blocks that is preserved by the automorphism group $\Aut(G)$ of $G$.
    For positive integers $s,t$, we say $G$ is {\it $(s,t)$-ring-like} with respect to $\vec{\sigma}$,  if each block of $\vec{\sigma}$ has size $s$ and any two adjacent vertices in $G$  belong to blocks that are at distance at most $t$ in the cyclic ordering of $\vec{\sigma}$.

 \begin{thm}[DeVos and Mohar, Theorem 1.9 in\cite{devos2006small}]\label{thm:devos sep}
    For every integer $k\geq 1$ and every connected vertex-transitive graph $G$ with $\mathrm{diam}(G)\geq 31k^2$, if a non-empty set $A\subseteq V(G)$ satisfies $|A|\leq |V(G)|/2$ , $|N(A)|=k$, and $G[A\cup N(A)]$ is connected, then one of the following holds.
    \begin{itemize}
        \item[(1)] $|A|\leq 2k^3+k^2$;
        \item[(2)] There exist positive integers $s,t$ with $st\leq k/2$ and a cyclic system $\vec{\sigma}$ on $G$ such that $G$ is $(s,t)$-ring-like with respect to $\vec{\sigma}$, and there exists an interval $J$ of $\vec{\sigma}$ with $Q:=\bigcup\limits_{B\in J}B$ satisfying $$Q \supseteq A\mbox{ \quad and \quad } |Q\setminus A|\leq \frac{k^3}{2}+k^2.$$   
    \end{itemize}
\end{thm}

The following observation can be verified directly.
\begin{obs}\label{obs:cycle merge}
Let \( X \) and \( Y \) be two cycles in a graph \( G \).  
Let \( \{P_1, \dots, P_r\} \) be a collection of pairwise disjoint subpaths of \( X \), and let \( \{Q_1, \dots, Q_r\} \) be a collection of pairwise disjoint subpaths of \( Y \) such that, for each \( i \in [r] \),  
\begin{itemize}
    \item \( P_i \) and \( Q_i \) have the same endpoints \( x_{2i-1}, x_{2i} \),  
    \item the vertices \( x_1, x_2, \ldots, x_{2r} \) appear in the same cyclic order along both \( X \) and \( Y \), and  
    \item each \( Q_i \) could only intersect \( X \) at vertices of the subpaths \( P_j \).
\end{itemize}
Then the subgraph obtained from \( X \) by replacing \( P_i \) with \( Q_i \) for each $i\in [r]$ (i.e., deleting all edges of \( P_i \) and adding all edges of \( Q_i \)) forms a cycle in \( G \).
\end{obs}

We now present the proof of Theorem~\ref{thm:transitive babai}, which uses the bound of Theorem~\ref{thm:lct upp}.
For a family $\mathcal{H}$ of subgraphs,  we define $V(\mathscr{H}) := \bigcup_{H \in \mathscr{H}} V(H)$ and $E(\mathscr{H}) := \bigcup_{H \in \mathscr{H}} E(H)$.

\begin{proof}[\bf Proof of Theorem~\ref{thm:transitive babai}.]
    Let $G$ be a connected $d$-regular vertex-transitive graph on $n$ vertices.
    Let $X,Y$ be any two longest cycles in $G$, and let $m:=|V(X)\cap V(Y)|\geq 1$. 
    Our goal is to show $m\geq \big(\log_d n\big)^{1/3}/100$.
    It suffices to assume that $d\geq 3$.
    
    Let $k$ denote the size of a minimum $(X,Y)$-separator in $G$. 
    Then there exists a subset $A\subseteq V(G)$ with $|A|\leq n/2$ such that $N(A)$ is a vertex cut of $G$ of size $k$ separating $X$ and $Y$, and $G[A\cup N(A)]$ contains exactly one of $X$ and $Y$ (without loss of generality, say $X$). 
    This implies 
    \begin{align}\label{equ:|A|-lower}
       |A|\geq |V(X)|-|N(A)|=c(G)-k.
    \end{align}
    Furthermore, we can assume that $G[A \cup N(A)]$ is connected, as otherwise we instead consider the intersection of $A$ with the component of $G[A \cup N(A)]$ that contains $X$.
    By the proof of Theorem~\ref{thm:lct upp}, we have $k\leq 10m^{3/2}+3m/2$. 
    If $k\geq \frac{\log_d^{1/2}n}{10}$, then $m\geq \frac{\log_d^{1/3}n}{100}$, which gives the desired bound. 
    We may therefore assume $k<\frac{\log_d^{1/2}n}{10}$ and equivalently, 
   \begin{align}\label{equ:n-lower}
          n> d^{100k^2}\geq 3^{100k^2}.  
    \end{align}  
    Since $G$ is $d$-regular, we have $\mathrm{diam}(G)\geq \log_{d-1}(n/d)\geq 31k^2$. 
    Applying Theorem~\ref{thm:devos sep} for $G$ and $A$, 
    we conclude that one of the items (1) or (2) holds.

First, suppose that item~(1) holds; that is, $|A|\leq 2k^3+k^2$. 
By Babai~\cite{babai1979long} we have $c(G)\geq 3\sqrt{n}$. 
Together with \eqref{equ:|A|-lower}, this yields
$3\sqrt{n}-k\leq c(G)-k\leq |A|\leq 2k^3+k^2$.
Hence, $n\leq \frac{1}{9}(2k^3+k^2+k)^2<3^{100k^2}$,
which contradicts \eqref{equ:n-lower}.
Therefore, from now on, we may assume that item (2) holds. 

    Then there exists a cyclic system $\vec{\sigma}$ on $G$ such that $G$ is $(s,t)$-ring-like with respect to $\vec{\sigma}$, where $st\leq k/2$, and there exists an interval $J$ of $\vec{\sigma}$ satisfying that $Q:=\bigcup_{B\in J}B\supseteq A$ and $|Q\setminus A|\leq \frac{k^3}{2}+k^2.$
    Let $B_1,\ldots,B_{n/s}$ be all blocks in the cyclic system $\vec{\sigma}$. 
    Without loss of generality, we may assume that for some $j\in [n/s]$, $V(X)\subseteq \bigcup_{i\in [j]}B_i$, 
    where $V(X)\cap V(B_1)\neq \emptyset$ and $V(X)\cap V(B_j)\neq \emptyset$.
    Since $V(X)\subseteq A\cup N(A)\subseteq \left(\bigcup_{B\in J}B\right)\cup N(A)$ and $|N(A)|=k$, we deduce that $|J|\geq j-k$.

    We first claim that $j \geq n/(2s)$. Suppose for a contradiction that $j < n/(2s)$. 
    By the vertex-transitivity of $G$, there is an automorphism $g\in \mathrm{Aut}(G)$ that maps $B_1$ to $B_{j+1}$. 
    Then the image $X^g$ is also a longest cycle in $G$, whose vertices belong to $\bigcup_{i\in [j+1,2j+1]}B_i$.
    Since $2j+1\leq n/s$, evidently we have $V(X)\cap V(X^g)=\emptyset$, which contradicts Proposition~\ref{prop:intersect}.
    This proves our claim that $j\geq n/(2s)$. In particular, $j\geq n/(2s) \geq 3^{100k^2}/k\geq 10^{k^2}$.

    We now proceed to establish the key step of the proof, namely that 
    \begin{align}\label{equ:key-ineq}
     j \geq n/s - 10^{k^2}.
    \end{align}
    Suppose, for the sake of contradiction, that this is not the case. 
    Then $10^{k^2}<3^{100k^2}/k\leq n/(2s)\leq j< n/s-10^{k^2}$.
    We aim to construct a cycle in $G$ which is strictly longer than $X$.
    
    For each $\ell\in [n/s-t+1]$, let $E_\ell$ denote the set of edges in $X$ incident to some vertex in $\bigcup_{i\in [\ell,\ell+t-1]}B_i$.
    Note that the edges of $E_\ell$ are decomposed into disjoint subpaths of $X$ with length at least two. 
    By the vertex-transitivity of $G$, for every $i\in [j]$, there is an automorphism $g_i\in \mathrm{Aut}(G)$ that maps $B_i$ to $B_1$.
    For each $\ell\in [j]$, we write $E_{\ell}^{g_\ell}$ for the image of $E_\ell$ under $g_\ell$, so that 
    \begin{align}\label{equ:E^g}
    \mbox{every edge in~} E_{\ell}^{g_\ell} \mbox{ is incident to some vertex in } \bigcup_{i\in [t]}B_i.
    \end{align}
    We define an equivalence relation $\sim$ among these automorphisms $g_1,\cdots,g_j$ by writing $g_u\sim g_v$ for some $u,v\in [j]$, if the following hold:
    \begin{itemize}
        \item[(a).] $E_u^{g_u}=E_v^{g_v}$ (as two edge sets incident to $\bigcup_{i\in [t]}B_i$);
        \item[(b).] The edges in $E_u^{g_u}$ appear in the same cyclic ordering along both longest cycles $X^{g_u}$ and $X^{g_v}$.
    \end{itemize}
    
    Now we calculate the number of equivalence classes under $\sim$.
    By \eqref{equ:E^g}, 
    every edge in $E_u^{g_u}$ has both endpoints in $B^*:=\bigcup_{i\in [-t+1,2t]}B_i$, where $B_{-i}:=B_{n/s-i}$ for $i\in [0,n/s-1]$.
    This implies $|E_u^{g_u}|\leq \binom{|B^*|}{2} \leq \binom{3st}{2}$ for every $u\in [j]$.
    We then consider the possible cyclic orderings for edges in $E_u^{g_u}$.
    As noted above, the edges in $E_u^{g_u}$ are decomposed into disjoint subpaths with length at least two, say $p$ such paths.
    Then $p\leq |B^*|/2=3st/2$.
    As each such path can be placed in 2 ways, 
    there are exactly $2^{p}(p-1)!<(2p)!\leq (3st)!$ possible cyclic orderings of $E_u^{g_u}$ in longest cycles  containing these edges.
    Therefore, the number of equivalence classes under $\sim$ is at most $$\sum\limits_{i=0}^{\binom{3st}{2}}\binom{\binom{3st}{2}}{i}\cdot (3st)!= 2^{\binom{3st}{2}}\cdot(3st)!\leq2^{\binom{3k/2}{2}}\cdot(3k/2)! <10^{k^2}.$$

    By the pigeonhole principle, there exist indices $1\leq u<v\leq 10^{k^2}\leq j$ such that $g_u\sim g_v$.
    Define the automorphism $g:=(g_v)^{-1}\circ g_u\in \mathrm{Aut}(G)$. Then we have 
    \[
    B_u^g=B_v \mbox{ \quad and \quad } E_u^g=E_v. 
    \]
    Since $j< n/s-10^{k^2}$, no edge connects $\bigcup_{i\in[u-1]}B_i$ and $\bigcup_{i\in[u+t,j]}B_i$.
    Hence, every edge in $E(X)\setminus E_u$ is entirely contained in either $G\big[\bigcup_{i\in[u-1]}B_i\big]$ or $G\big[\bigcup_{i\in[u+t,j]}B_i\big]$.
    These edges form several disjoint subpaths of $X$.
    Define $\mathcal{R}$ to be the collection of disjoint paths formed by $E(X)\setminus E_u$ within $G[\bigcup_{i\in[u+t,j]}B_i]$, and define $\mathcal{R}'$ to be the collection of disjoint paths formed by $E(X)\setminus E_v$ within $G[\bigcup_{i\in[v+t,j]}B_i]$;
    see Figure~\ref{fig:thm9} for an illustration.
Then we have $E(\mathcal{R}')\subseteq E(\mathcal{R})$.

    We claim that $E(\mathcal{R})\setminus E(\mathcal{R}')\neq \emptyset$.
    Select an edge $e\in E_u$ whose endpoints belong to $B_\alpha$ and $B_\beta$ respectively, such that $\min\{\alpha,\beta\}$ is maximized. 
    Then the edge $e^{g}\in E_u^{g}=E_v$, hence $e^{g}\notin E(\mathcal{R}')$.
    Note that $e^g$ is incident to blocks $B_{\alpha+v-u}$ and $B_{\beta+v-u}$. 
    We assert that $e^g \notin E_u$ and thus, $e^g \in E(\mathcal{R})$. 
    Indeed, if $e^g\in E_u$, then $\min\{\alpha+v-u, \beta+v-u\} > \min\{\alpha, \beta\}$ would contradict the maximality of $\min\{\alpha,\beta\}$ for our choice of $e$.
    We have demonstrated that $e^g \in E(\mathcal{R}) \setminus E(\mathcal{R}')$, and the claim follows.
    This implies $|E(\mathcal{R})| > |E(\mathcal{R}')|$.

It is clear that the endpoints of subpaths (of $X$) formed by $E_v$ is the same as that of $E(X)\setminus E_v$;\footnote{Note that these subpaths may be viewed in different cycles, but these endpoints are determined solely by the edge set $E_v$, and the same holds for other edge sets.}
similarly, the endpoints of subpaths (of $X$) formed by $E_u$ is the same as that of $E(X)\setminus E_u$, 
which says equivalently, the endpoints of subpaths (of $X^g$) formed by $E_v=E_u^g$ is the same as that of $\left(E(X)\setminus E_u\right)^g$.
Summarizing, the endpoints of subpaths formed by $E_v$, by $E(X)\setminus E_v$, and by $\left(E(X)\setminus E_u\right)^g$ are all the same.
Let $I$ denote the set of these endpoints that belong to $\bigcup_{i\in [v+t,j]}B_i$.

Recall the definitions of $\mathcal{R}$ and $\mathcal{R}'$. 
Let $\mathcal{R}^g$ be the image of $\mathcal{R}$ under the automorphism $g$.
Putting everything together, we can derive that 
the endpoints of the paths in $\mathcal{R}'$ are the same as those of $\mathcal{R}^g$, that is, the vertices of $I$.
Moreover, since $g_u\sim g_v$, 
these vertices of $I$ appear in the same cyclic order along both $X$ and $X^g$. 
Therefore, we can label the vertices as $I=\{x_1,x_2,\cdots,x_{2r}\}$ and write $\mathcal{R}'=\{P_1,\cdots,P_r\}$ and $\mathcal{R}^g=\{Q_1,\cdots,Q_r\}$, such that $P_i$ and $Q_i$ share the same endpoints $\{x_{2i-1},x_{2i}\}$ for every $i\in [r]$.

We view all $P_i\in \mathcal{R}'$ as subpaths of the cycle $X$ and all $Q_i\in \mathcal{R}^g$ as subpaths of the cycle $X^g$.
To apply Observation~\ref{obs:cycle merge},
it remains to verify that each $Q_i$ could only intersect $X$ at vertices of some $P_j$.
This follows since all vertices of $Q_i\in \mathcal{R}^g$ are entirely contained in $\bigcup_{j\in [v+t,j+v-u]}B_j$, whereas the vertices in $V(X)\setminus V(\mathcal{R}')$ must belong to $\bigcup_{j\in [v+t-1]}B_j$.
Since $j+(v-u) < n/s - 10^{k^2}+(v-u)< n/s$,
these two block intervals are disjoint, which completes the verification.
Finally, we can apply Observation~\ref{obs:cycle merge} to conclude that the edges in $E(X)\cup E(\mathcal{R}^g)\setminus E(\mathcal{R}')$ form a cycle in $G$ of length $|X|+|E(\mathcal{R})|-|E(\mathcal{R}')|>|X|$, which contradicts the maximality of $X$.
This proves \eqref{equ:key-ineq}, namely, $j>n/s-10^{k^2}$.
Hence, we have $|J|\geq j-k\geq n/s-10^{k^2}-k$.
    
Recall that $Q=\bigcup_{B\in J}B$. 
Since $|Q\setminus A|\leq \frac{k^3}{2}+k^2$, we obtain the following inequality
    $$\frac{n}{2}\geq |A|\geq |Q|-\left(\frac{k^3}{2}+k^2\right)=s|J|-\left(\frac{k^3}{2}+k^2\right)\geq s(n/s-10^{k^2}-k)-\left(\frac{k^3}{2}+k^2\right)>n-10^{k^2}\cdot k,$$
which gives $n < 2k\cdot 10^{k^2} < 3^{100k^2}$, contradicting \eqref{equ:n-lower}, and completes the proof of Theorem~\ref{thm:transitive babai}.
\end{proof}

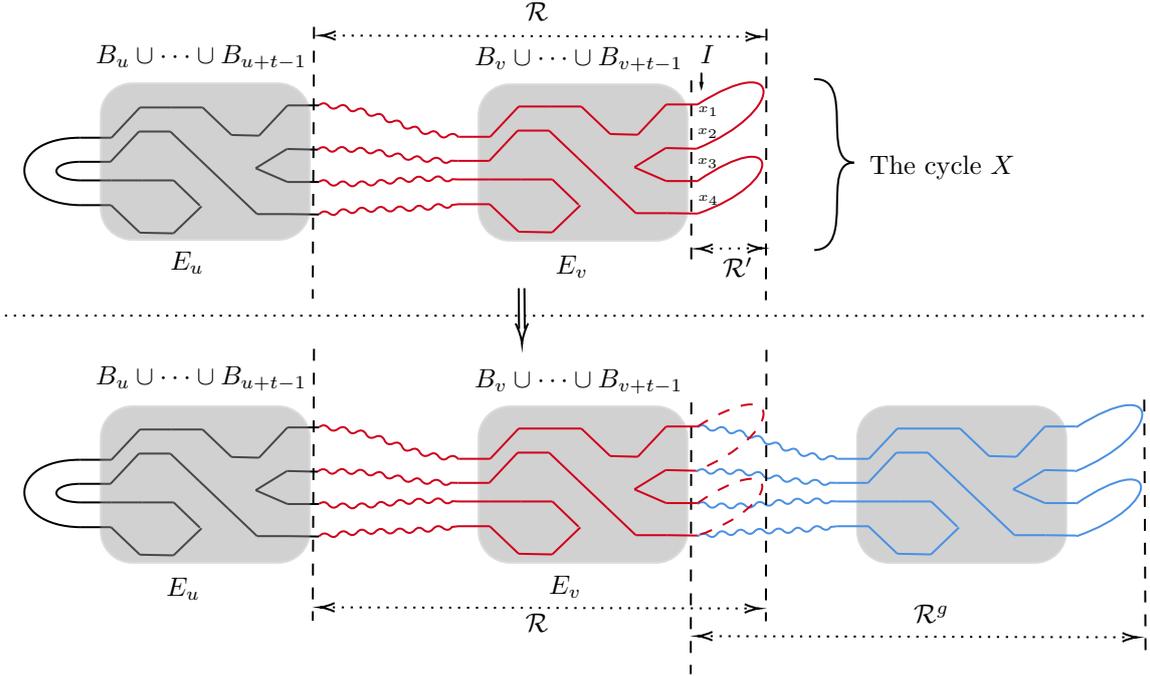
\begin{figure}
\captionsetup{justification=centering}
    \centering

\tikzset{every picture/.style={line width=0.75pt}} 
\begin{tikzpicture}[x=0.75pt,y=0.75pt,yscale=-1,xscale=1,scale = 0.95]

\draw    (93.14,117.32) -- (108.09,101.29) ;
\draw    (92.89,153.02) -- (108.09,167.6) ;
\draw    (92.64,130.07) -- (107.58,114.04) ;
\draw    (108.09,167.6) -- (125.73,167.95) ;
\draw    (109.36,139.91) -- (125.83,139.91) ;
\draw    (125.73,167.95) -- (140.5,153.63) ;
\draw    (125.83,139.91) -- (139.59,153.27) ;
\draw    (108.09,101.29) -- (124.56,101.29) ;
\draw    (107.58,114.04) -- (124.05,114.04) ;
\draw    (124.05,114.04) -- (139.25,128.62) ;
\draw    (124.56,101.29) -- (141.03,101.29) ;
\draw    (141.03,101.29) -- (156.23,115.86) ;
\draw    (139.25,128.62) -- (154.46,143.19) ;
\draw    (154.46,143.19) -- (169.66,157.76) ;
\draw    (156.23,115.86) -- (170.67,116.23) ;
\draw    (170.67,116.23) -- (185.62,100.2) ;
\draw    (185.62,100.2) -- (202.09,100.2) ;
\draw    (169.66,157.76) -- (186.13,157.76) ;
\draw    (186.13,157.76) -- (202.17,157.97) ;
\draw    (185.37,123.27) -- (201.5,123.64) ;
\draw    (185.62,140.76) -- (202.09,140.76) ;
\draw    (168.9,133.47) -- (185.37,123.27) ;
\draw    (168.9,133.47) -- (185.62,140.76) ;
\draw  [dash pattern={on 4.5pt off 4.5pt}]  (199.56,58.61) -- (198.88,202.57) ;
\draw [color={rgb, 255:red, 0; green, 0; blue, 0 }  ,draw opacity=1 ] [dash pattern={on 4.5pt off 4.5pt}]  (397.8,86.98) -- (397.8,185.02) ;
\draw    (310.12,196.98) -- (310.29,218.56)(307.12,197) -- (307.29,218.58) ;
\draw [shift={(308.85,226.57)}, rotate = 269.56] [color={rgb, 255:red, 0; green, 0; blue, 0 }  ][line width=0.75]    (10.93,-3.29) .. controls (6.95,-1.4) and (3.31,-0.3) .. (0,0) .. controls (3.31,0.3) and (6.95,1.4) .. (10.93,3.29)   ;
\draw [line width=0.75]  [dash pattern={on 0.84pt off 2.51pt}]  (205.21,63.21) -- (432,63.83) ;
\draw [shift={(434,63.83)}, rotate = 180.16] [color={rgb, 255:red, 0; green, 0; blue, 0 }  ][line width=0.75]    (7.65,-2.3) .. controls (4.86,-0.97) and (2.31,-0.21) .. (0,0) .. controls (2.31,0.21) and (4.86,0.98) .. (7.65,2.3)   ;
\draw [shift={(203.21,63.2)}, rotate = 0.16] [color={rgb, 255:red, 0; green, 0; blue, 0 }  ][line width=0.75]    (7.65,-2.3) .. controls (4.86,-0.97) and (2.31,-0.21) .. (0,0) .. controls (2.31,0.21) and (4.86,0.98) .. (7.65,2.3)   ;
\draw [line width=0.75]  [dash pattern={on 0.84pt off 2.51pt}]  (403.36,176.08) -- (433,175.85) ;
\draw [shift={(435,175.83)}, rotate = 179.55] [color={rgb, 255:red, 0; green, 0; blue, 0 }  ][line width=0.75]    (7.65,-2.3) .. controls (4.86,-0.97) and (2.31,-0.21) .. (0,0) .. controls (2.31,0.21) and (4.86,0.98) .. (7.65,2.3)   ;
\draw [shift={(401.36,176.1)}, rotate = 359.55] [color={rgb, 255:red, 0; green, 0; blue, 0 }  ][line width=0.75]    (7.65,-2.3) .. controls (4.86,-0.97) and (2.31,-0.21) .. (0,0) .. controls (2.31,0.21) and (4.86,0.98) .. (7.65,2.3)   ;
\draw  [dash pattern={on 4.5pt off 4.5pt}]  (437.29,59.62) -- (437,203.33) ;
\draw    (76.67,117.32) -- (93.14,117.32) ;
\draw    (76.17,130.07) -- (92.64,130.07) ;
\draw    (76.42,139.91) -- (92.89,139.91) ;
\draw    (76.42,153.02) -- (92.89,153.02) ;
\draw    (92.89,139.91) -- (109.36,139.91) ;
\draw  [color={rgb, 255:red, 155; green, 155; blue, 155 }  ,draw opacity=0.36 ][fill={rgb, 255:red, 155; green, 155; blue, 155 }  ,fill opacity=0.46 ] (286.06,105.63) .. controls (286.06,96.48) and (293.47,89.07) .. (302.61,89.07) -- (379.22,89.07) .. controls (388.36,89.07) and (395.77,96.48) .. (395.77,105.63) -- (395.77,155.28) .. controls (395.77,164.42) and (388.36,171.83) .. (379.22,171.83) -- (302.61,171.83) .. controls (293.47,171.83) and (286.06,164.42) .. (286.06,155.28) -- cycle ;
\draw [color={rgb, 255:red, 208; green, 2; blue, 27 }  ,draw opacity=1 ]   (292.14,116.82) -- (307.09,100.79) ;
\draw [color={rgb, 255:red, 208; green, 2; blue, 27 }  ,draw opacity=1 ]   (291.89,152.52) -- (307.09,167.1) ;
\draw [color={rgb, 255:red, 208; green, 2; blue, 27 }  ,draw opacity=1 ]   (291.64,129.57) -- (306.58,113.54) ;
\draw [color={rgb, 255:red, 208; green, 2; blue, 27 }  ,draw opacity=1 ]   (307.09,167.1) -- (324.73,167.45) ;
\draw [color={rgb, 255:red, 208; green, 2; blue, 27 }  ,draw opacity=1 ]   (308.36,139.41) -- (324.83,139.41) ;
\draw [color={rgb, 255:red, 208; green, 2; blue, 27 }  ,draw opacity=1 ]   (324.73,167.45) -- (339.5,153.13) ;
\draw [color={rgb, 255:red, 208; green, 2; blue, 27 }  ,draw opacity=1 ]   (324.83,139.41) -- (338.59,152.77) ;
\draw [color={rgb, 255:red, 208; green, 2; blue, 27 }  ,draw opacity=1 ]   (307.09,100.79) -- (323.56,100.79) ;
\draw [color={rgb, 255:red, 208; green, 2; blue, 27 }  ,draw opacity=1 ]   (306.58,113.54) -- (323.05,113.54) ;
\draw [color={rgb, 255:red, 208; green, 2; blue, 27 }  ,draw opacity=1 ]   (323.05,113.54) -- (338.25,128.12) ;
\draw [color={rgb, 255:red, 208; green, 2; blue, 27 }  ,draw opacity=1 ]   (323.56,100.79) -- (340.03,100.79) ;
\draw [color={rgb, 255:red, 208; green, 2; blue, 27 }  ,draw opacity=1 ]   (340.03,100.79) -- (355.23,115.36) ;
\draw [color={rgb, 255:red, 208; green, 2; blue, 27 }  ,draw opacity=1 ]   (338.25,128.12) -- (353.46,142.69) ;
\draw [color={rgb, 255:red, 208; green, 2; blue, 27 }  ,draw opacity=1 ]   (353.46,142.69) -- (368.66,157.26) ;
\draw [color={rgb, 255:red, 208; green, 2; blue, 27 }  ,draw opacity=1 ]   (355.23,115.36) -- (369.67,115.73) ;
\draw [color={rgb, 255:red, 208; green, 2; blue, 27 }  ,draw opacity=1 ]   (369.67,115.73) -- (384.62,99.7) ;
\draw [color={rgb, 255:red, 208; green, 2; blue, 27 }  ,draw opacity=1 ]   (384.62,99.7) -- (401.09,99.7) ;
\draw [color={rgb, 255:red, 208; green, 2; blue, 27 }  ,draw opacity=1 ]   (368.66,157.26) -- (385.13,157.26) ;
\draw [color={rgb, 255:red, 208; green, 2; blue, 27 }  ,draw opacity=1 ]   (385.13,157.26) -- (401.17,157.47) ;
\draw [color={rgb, 255:red, 208; green, 2; blue, 27 }  ,draw opacity=1 ]   (384.37,122.77) -- (400.5,123.14) ;
\draw [color={rgb, 255:red, 208; green, 2; blue, 27 }  ,draw opacity=1 ]   (384.62,140.26) -- (401.09,140.26) ;
\draw [color={rgb, 255:red, 208; green, 2; blue, 27 }  ,draw opacity=1 ]   (367.9,132.97) -- (384.37,122.77) ;
\draw [color={rgb, 255:red, 208; green, 2; blue, 27 }  ,draw opacity=1 ]   (367.9,132.97) -- (384.62,140.26) ;
\draw [color={rgb, 255:red, 208; green, 2; blue, 27 }  ,draw opacity=1 ]   (275.67,116.82) -- (292.14,116.82) ;
\draw [color={rgb, 255:red, 208; green, 2; blue, 27 }  ,draw opacity=1 ]   (275.17,129.57) -- (291.64,129.57) ;
\draw [color={rgb, 255:red, 208; green, 2; blue, 27 }  ,draw opacity=1 ]   (275.42,139.41) -- (291.89,139.41) ;
\draw [color={rgb, 255:red, 208; green, 2; blue, 27 }  ,draw opacity=1 ]   (275.42,152.52) -- (291.89,152.52) ;
\draw [color={rgb, 255:red, 208; green, 2; blue, 27 }  ,draw opacity=1 ]   (291.89,139.41) -- (308.36,139.41) ;
\draw [color={rgb, 255:red, 208; green, 2; blue, 27 }  ,draw opacity=1 ]   (202.09,100.2) .. controls (203.92,98.61) and (205.62,98.69) .. (207.2,100.46) .. controls (208.65,102.25) and (210.35,102.4) .. (212.32,100.92) .. controls (214.29,99.49) and (215.98,99.73) .. (217.4,101.63) .. controls (218.66,103.55) and (220.2,103.84) .. (222.01,102.51) .. controls (224.08,101.28) and (225.72,101.69) .. (226.91,103.73) .. controls (227.96,105.77) and (229.56,106.26) .. (231.73,105.2) .. controls (233.91,104.19) and (235.52,104.75) .. (236.55,106.9) .. controls (237.52,109.05) and (239.07,109.64) .. (241.22,108.67) .. controls (243.37,107.7) and (244.93,108.29) .. (245.92,110.44) .. controls (246.79,112.52) and (248.33,113.05) .. (250.54,112.04) .. controls (252.53,110.92) and (254.08,111.38) .. (255.18,113.43) .. controls (256.43,115.48) and (258.12,115.89) .. (260.25,114.67) .. controls (262.12,113.35) and (263.71,113.65) .. (265.04,115.58) .. controls (266.53,117.49) and (268.13,117.72) .. (269.86,116.27) .. controls (271.91,114.81) and (273.68,114.98) .. (275.17,116.78) -- (275.67,116.82) ;
\draw [color={rgb, 255:red, 208; green, 2; blue, 27 }  ,draw opacity=1 ]   (201.5,123.64) .. controls (203.48,121.91) and (205.25,121.87) .. (206.82,123.51) .. controls (208.29,125.16) and (209.87,125.14) .. (211.57,123.46) .. controls (213.51,121.81) and (215.21,121.83) .. (216.68,123.52) .. controls (218.29,125.23) and (219.89,125.29) .. (221.5,123.7) .. controls (223.29,122.15) and (225,122.27) .. (226.62,124.08) .. controls (228.03,125.91) and (229.66,126.1) .. (231.52,124.67) .. controls (233.51,123.3) and (235.18,123.58) .. (236.53,125.52) .. controls (237.8,127.46) and (239.41,127.77) .. (241.34,126.45) .. controls (243.32,125.13) and (244.97,125.43) .. (246.29,127.35) .. controls (247.8,129.28) and (249.52,129.53) .. (251.46,128.11) .. controls (253.26,126.64) and (254.82,126.81) .. (256.15,128.62) .. controls (257.92,130.45) and (259.61,130.58) .. (261.2,129.02) .. controls (263.17,127.47) and (264.96,127.57) .. (266.55,129.32) .. controls (267.92,131.04) and (269.57,131.1) .. (271.52,129.49) -- (275.17,129.57) ;
\draw [color={rgb, 255:red, 208; green, 2; blue, 27 }  ,draw opacity=1 ]   (202.09,140.76) .. controls (204.15,139.29) and (205.86,139.46) .. (207.22,141.26) .. controls (208.84,143.05) and (210.53,143.17) .. (212.29,141.6) .. controls (213.92,140.01) and (215.58,140.07) .. (217.26,141.78) .. controls (218.93,143.45) and (220.55,143.45) .. (222.12,141.76) .. controls (223.78,140.04) and (225.48,139.97) .. (227.21,141.55) .. controls (228.95,143.1) and (230.64,142.96) .. (232.27,141.15) .. controls (233.64,139.34) and (235.28,139.17) .. (237.19,140.65) .. controls (238.92,142.14) and (240.59,141.96) .. (242.2,140.11) .. controls (243.68,138.29) and (245.35,138.13) .. (247.2,139.64) .. controls (248.87,141.18) and (250.46,141.07) .. (251.96,139.31) .. controls (253.59,137.56) and (255.34,137.48) .. (257.2,139.08) .. controls (258.84,140.71) and (260.5,140.69) .. (262.17,139.02) .. controls (263.98,137.36) and (265.66,137.38) .. (267.2,139.08) .. controls (268.77,140.79) and (270.44,140.84) .. (272.2,139.25) -- (275.42,139.41) ;
\draw [color={rgb, 255:red, 208; green, 2; blue, 27 }  ,draw opacity=1 ]   (202.17,157.97) .. controls (203.78,156.07) and (205.55,155.84) .. (207.49,157.27) .. controls (209.32,158.72) and (210.92,158.53) .. (212.27,156.7) .. controls (213.85,154.85) and (215.55,154.68) .. (217.36,156.17) .. controls (219.05,157.69) and (220.66,157.55) .. (222.21,155.75) .. controls (223.68,153.98) and (225.31,153.87) .. (227.11,155.44) .. controls (228.96,157.02) and (230.6,156.95) .. (232.03,155.24) .. controls (233.8,153.53) and (235.48,153.5) .. (237.09,155.14) .. controls (238.86,156.78) and (240.53,156.76) .. (242.11,155.07) .. controls (243.72,153.36) and (245.43,153.32) .. (247.22,154.94) .. controls (248.95,156.53) and (250.65,156.45) .. (252.34,154.7) .. controls (253.89,152.93) and (255.47,152.83) .. (257.1,154.39) .. controls (258.85,155.92) and (260.52,155.78) .. (262.09,153.97) .. controls (263.76,152.14) and (265.49,151.97) .. (267.27,153.46) .. controls (269.17,154.93) and (270.75,154.75) .. (272,152.93) -- (275.42,152.52) ;
\draw    (76.42,153.02) .. controls (38.67,152.33) and (36.67,117.33) .. (76.67,117.32) ;
\draw    (76.17,130.07) .. controls (60.67,130.33) and (59.67,139.33) .. (76.42,139.91) ;
\draw [color={rgb, 255:red, 208; green, 2; blue, 27 }  ,draw opacity=1 ]   (401.09,99.7) .. controls (447.67,70) and (447.17,104) .. (400.5,123.14) ;
\draw [color={rgb, 255:red, 208; green, 2; blue, 27 }  ,draw opacity=1 ]   (401.09,140.26) .. controls (445.67,109.5) and (447.17,141.5) .. (401.17,157.47) ;
\draw    (93.14,287.82) -- (108.09,271.79) ;
\draw    (92.89,323.52) -- (108.09,338.1) ;
\draw    (92.64,300.57) -- (107.58,284.54) ;
\draw    (108.09,338.1) -- (125.73,338.45) ;
\draw    (109.36,310.41) -- (125.83,310.41) ;
\draw    (125.73,338.45) -- (140.5,324.13) ;
\draw    (125.83,310.41) -- (139.59,323.77) ;
\draw    (108.09,271.79) -- (124.56,271.79) ;
\draw    (107.58,284.54) -- (124.05,284.54) ;
\draw    (124.05,284.54) -- (139.25,299.12) ;
\draw    (124.56,271.79) -- (141.03,271.79) ;
\draw    (141.03,271.79) -- (156.23,286.36) ;
\draw    (139.25,299.12) -- (154.46,313.69) ;
\draw    (154.46,313.69) -- (169.66,328.26) ;
\draw    (156.23,286.36) -- (170.67,286.73) ;
\draw    (170.67,286.73) -- (185.62,270.7) ;
\draw    (185.62,270.7) -- (202.09,270.7) ;
\draw    (169.66,328.26) -- (186.13,328.26) ;
\draw    (186.13,328.26) -- (202.17,328.47) ;
\draw    (185.37,293.77) -- (201.5,294.14) ;
\draw    (185.62,311.26) -- (202.09,311.26) ;
\draw    (168.9,303.97) -- (185.37,293.77) ;
\draw    (168.9,303.97) -- (185.62,311.26) ;
\draw  [dash pattern={on 4.5pt off 4.5pt}]  (199.56,229.11) -- (198.88,373.07) ;
\draw [color={rgb, 255:red, 0; green, 0; blue, 0 }  ,draw opacity=1 ] [dash pattern={on 4.5pt off 4.5pt}]  (397.8,257.48) -- (397.5,401.33) ;
\draw    (76.67,287.82) -- (93.14,287.82) ;
\draw    (76.17,300.57) -- (92.64,300.57) ;
\draw    (76.42,310.41) -- (92.89,310.41) ;
\draw    (76.42,323.52) -- (92.89,323.52) ;
\draw    (92.89,310.41) -- (109.36,310.41) ;
\draw  [color={rgb, 255:red, 155; green, 155; blue, 155 }  ,draw opacity=0.36 ][fill={rgb, 255:red, 155; green, 155; blue, 155 }  ,fill opacity=0.46 ] (286.06,276.13) .. controls (286.06,266.98) and (293.47,259.57) .. (302.61,259.57) -- (379.22,259.57) .. controls (388.36,259.57) and (395.77,266.98) .. (395.77,276.13) -- (395.77,325.78) .. controls (395.77,334.92) and (388.36,342.33) .. (379.22,342.33) -- (302.61,342.33) .. controls (293.47,342.33) and (286.06,334.92) .. (286.06,325.78) -- cycle ;
\draw [color={rgb, 255:red, 208; green, 2; blue, 27 }  ,draw opacity=1 ]   (292.14,287.32) -- (307.09,271.29) ;
\draw [color={rgb, 255:red, 208; green, 2; blue, 27 }  ,draw opacity=1 ]   (291.89,323.02) -- (307.09,337.6) ;
\draw [color={rgb, 255:red, 208; green, 2; blue, 27 }  ,draw opacity=1 ]   (291.64,300.07) -- (306.58,284.04) ;
\draw [color={rgb, 255:red, 208; green, 2; blue, 27 }  ,draw opacity=1 ]   (307.09,337.6) -- (324.73,337.95) ;
\draw [color={rgb, 255:red, 208; green, 2; blue, 27 }  ,draw opacity=1 ]   (308.36,309.91) -- (324.83,309.91) ;
\draw [color={rgb, 255:red, 208; green, 2; blue, 27 }  ,draw opacity=1 ]   (324.73,337.95) -- (339.5,323.63) ;
\draw [color={rgb, 255:red, 208; green, 2; blue, 27 }  ,draw opacity=1 ]   (324.83,309.91) -- (338.59,323.27) ;
\draw [color={rgb, 255:red, 208; green, 2; blue, 27 }  ,draw opacity=1 ]   (307.09,271.29) -- (323.56,271.29) ;
\draw [color={rgb, 255:red, 208; green, 2; blue, 27 }  ,draw opacity=1 ]   (306.58,284.04) -- (323.05,284.04) ;
\draw [color={rgb, 255:red, 208; green, 2; blue, 27 }  ,draw opacity=1 ]   (323.05,284.04) -- (338.25,298.62) ;
\draw [color={rgb, 255:red, 208; green, 2; blue, 27 }  ,draw opacity=1 ]   (323.56,271.29) -- (340.03,271.29) ;
\draw [color={rgb, 255:red, 208; green, 2; blue, 27 }  ,draw opacity=1 ]   (340.03,271.29) -- (355.23,285.86) ;
\draw [color={rgb, 255:red, 208; green, 2; blue, 27 }  ,draw opacity=1 ]   (338.25,298.62) -- (353.46,313.19) ;
\draw [color={rgb, 255:red, 208; green, 2; blue, 27 }  ,draw opacity=1 ]   (353.46,313.19) -- (368.66,327.76) ;
\draw [color={rgb, 255:red, 208; green, 2; blue, 27 }  ,draw opacity=1 ]   (355.23,285.86) -- (369.67,286.23) ;
\draw [color={rgb, 255:red, 208; green, 2; blue, 27 }  ,draw opacity=1 ]   (369.67,286.23) -- (384.62,270.2) ;
\draw [color={rgb, 255:red, 208; green, 2; blue, 27 }  ,draw opacity=1 ]   (384.62,270.2) -- (401.09,270.2) ;
\draw [color={rgb, 255:red, 208; green, 2; blue, 27 }  ,draw opacity=1 ]   (368.66,327.76) -- (385.13,327.76) ;
\draw [color={rgb, 255:red, 208; green, 2; blue, 27 }  ,draw opacity=1 ]   (385.13,327.76) -- (401.17,327.97) ;
\draw [color={rgb, 255:red, 208; green, 2; blue, 27 }  ,draw opacity=1 ]   (384.37,293.27) -- (400.5,293.64) ;
\draw [color={rgb, 255:red, 208; green, 2; blue, 27 }  ,draw opacity=1 ]   (384.62,310.76) -- (401.09,310.76) ;
\draw [color={rgb, 255:red, 208; green, 2; blue, 27 }  ,draw opacity=1 ]   (367.9,303.47) -- (384.37,293.27) ;
\draw [color={rgb, 255:red, 208; green, 2; blue, 27 }  ,draw opacity=1 ]   (367.9,303.47) -- (384.62,310.76) ;
\draw [color={rgb, 255:red, 208; green, 2; blue, 27 }  ,draw opacity=1 ]   (275.67,287.32) -- (292.14,287.32) ;
\draw [color={rgb, 255:red, 208; green, 2; blue, 27 }  ,draw opacity=1 ]   (275.17,300.07) -- (291.64,300.07) ;
\draw [color={rgb, 255:red, 208; green, 2; blue, 27 }  ,draw opacity=1 ]   (275.42,309.91) -- (291.89,309.91) ;
\draw [color={rgb, 255:red, 208; green, 2; blue, 27 }  ,draw opacity=1 ]   (275.42,323.02) -- (291.89,323.02) ;
\draw [color={rgb, 255:red, 208; green, 2; blue, 27 }  ,draw opacity=1 ]   (291.89,309.91) -- (308.36,309.91) ;
\draw [color={rgb, 255:red, 208; green, 2; blue, 27 }  ,draw opacity=1 ]   (202.09,270.7) .. controls (203.92,269.11) and (205.62,269.19) .. (207.2,270.96) .. controls (208.65,272.75) and (210.35,272.9) .. (212.32,271.42) .. controls (214.29,269.99) and (215.98,270.23) .. (217.4,272.13) .. controls (218.66,274.05) and (220.2,274.34) .. (222.01,273.01) .. controls (224.08,271.78) and (225.72,272.19) .. (226.91,274.23) .. controls (227.96,276.27) and (229.56,276.76) .. (231.73,275.7) .. controls (233.91,274.69) and (235.52,275.25) .. (236.55,277.4) .. controls (237.52,279.55) and (239.07,280.14) .. (241.22,279.17) .. controls (243.37,278.2) and (244.93,278.79) .. (245.92,280.94) .. controls (246.79,283.02) and (248.33,283.55) .. (250.54,282.54) .. controls (252.53,281.42) and (254.08,281.88) .. (255.18,283.93) .. controls (256.43,285.98) and (258.12,286.39) .. (260.25,285.17) .. controls (262.12,283.85) and (263.71,284.15) .. (265.04,286.08) .. controls (266.53,287.99) and (268.13,288.22) .. (269.86,286.77) .. controls (271.91,285.31) and (273.68,285.48) .. (275.17,287.28) -- (275.67,287.32) ;
\draw [color={rgb, 255:red, 208; green, 2; blue, 27 }  ,draw opacity=1 ]   (201.5,294.14) .. controls (203.48,292.41) and (205.25,292.37) .. (206.82,294.01) .. controls (208.29,295.66) and (209.87,295.64) .. (211.57,293.96) .. controls (213.51,292.31) and (215.21,292.33) .. (216.68,294.02) .. controls (218.29,295.73) and (219.89,295.79) .. (221.5,294.2) .. controls (223.29,292.65) and (225,292.77) .. (226.62,294.58) .. controls (228.03,296.41) and (229.66,296.6) .. (231.52,295.17) .. controls (233.51,293.8) and (235.18,294.08) .. (236.53,296.02) .. controls (237.8,297.96) and (239.41,298.27) .. (241.34,296.95) .. controls (243.32,295.63) and (244.97,295.93) .. (246.29,297.85) .. controls (247.8,299.78) and (249.52,300.03) .. (251.46,298.61) .. controls (253.26,297.14) and (254.82,297.31) .. (256.15,299.12) .. controls (257.92,300.95) and (259.61,301.08) .. (261.2,299.52) .. controls (263.17,297.97) and (264.96,298.07) .. (266.55,299.82) .. controls (267.92,301.54) and (269.57,301.6) .. (271.52,299.99) -- (275.17,300.07) ;
\draw [color={rgb, 255:red, 208; green, 2; blue, 27 }  ,draw opacity=1 ]   (202.09,311.26) .. controls (204.15,309.79) and (205.86,309.96) .. (207.22,311.76) .. controls (208.84,313.55) and (210.53,313.67) .. (212.29,312.1) .. controls (213.92,310.51) and (215.58,310.57) .. (217.26,312.28) .. controls (218.93,313.95) and (220.55,313.95) .. (222.12,312.26) .. controls (223.78,310.54) and (225.48,310.47) .. (227.21,312.05) .. controls (228.95,313.6) and (230.64,313.46) .. (232.27,311.65) .. controls (233.64,309.84) and (235.28,309.67) .. (237.19,311.15) .. controls (238.92,312.64) and (240.59,312.46) .. (242.2,310.61) .. controls (243.68,308.79) and (245.35,308.63) .. (247.2,310.14) .. controls (248.87,311.68) and (250.46,311.57) .. (251.96,309.81) .. controls (253.59,308.06) and (255.34,307.98) .. (257.2,309.58) .. controls (258.84,311.21) and (260.5,311.19) .. (262.17,309.52) .. controls (263.98,307.86) and (265.66,307.88) .. (267.2,309.58) .. controls (268.77,311.29) and (270.44,311.34) .. (272.2,309.75) -- (275.42,309.91) ;
\draw [color={rgb, 255:red, 208; green, 2; blue, 27 }  ,draw opacity=1 ]   (202.17,328.47) .. controls (203.78,326.57) and (205.55,326.34) .. (207.49,327.77) .. controls (209.32,329.22) and (210.92,329.03) .. (212.27,327.2) .. controls (213.85,325.35) and (215.55,325.18) .. (217.36,326.67) .. controls (219.05,328.19) and (220.66,328.05) .. (222.21,326.25) .. controls (223.68,324.48) and (225.31,324.37) .. (227.11,325.94) .. controls (228.96,327.52) and (230.6,327.45) .. (232.03,325.74) .. controls (233.8,324.03) and (235.48,324) .. (237.09,325.64) .. controls (238.86,327.28) and (240.53,327.26) .. (242.11,325.57) .. controls (243.72,323.86) and (245.43,323.82) .. (247.22,325.44) .. controls (248.95,327.03) and (250.65,326.95) .. (252.34,325.2) .. controls (253.89,323.43) and (255.47,323.33) .. (257.1,324.89) .. controls (258.85,326.42) and (260.52,326.28) .. (262.09,324.47) .. controls (263.76,322.64) and (265.49,322.47) .. (267.27,323.96) .. controls (269.17,325.43) and (270.75,325.25) .. (272,323.43) -- (275.42,323.02) ;
\draw    (76.42,323.52) .. controls (38.67,322.83) and (36.67,287.83) .. (76.67,287.82) ;
\draw    (76.17,300.57) .. controls (60.67,300.83) and (59.67,309.83) .. (76.42,310.41) ;
\draw [color={rgb, 255:red, 0; green, 0; blue, 0 }  ,draw opacity=1 ] [dash pattern={on 4.5pt off 4.5pt}]  (635.8,255.98) -- (636.5,394.83) ;
\draw  [color={rgb, 255:red, 155; green, 155; blue, 155 }  ,draw opacity=0.39 ][fill={rgb, 255:red, 155; green, 155; blue, 155 }  ,fill opacity=0.46 ] (485.06,276.13) .. controls (485.06,266.98) and (492.47,259.57) .. (501.61,259.57) -- (578.22,259.57) .. controls (587.36,259.57) and (594.77,266.98) .. (594.77,276.13) -- (594.77,325.78) .. controls (594.77,334.92) and (587.36,342.33) .. (578.22,342.33) -- (501.61,342.33) .. controls (492.47,342.33) and (485.06,334.92) .. (485.06,325.78) -- cycle ;
\draw [color={rgb, 255:red, 74; green, 144; blue, 226 }  ,draw opacity=1 ]   (491.14,287.32) -- (506.09,271.29) ;
\draw [color={rgb, 255:red, 74; green, 144; blue, 226 }  ,draw opacity=1 ]   (490.89,323.02) -- (506.09,337.6) ;
\draw [color={rgb, 255:red, 74; green, 144; blue, 226 }  ,draw opacity=1 ]   (490.64,300.07) -- (505.58,284.04) ;
\draw [color={rgb, 255:red, 74; green, 144; blue, 226 }  ,draw opacity=1 ]   (506.09,337.6) -- (523.73,337.95) ;
\draw [color={rgb, 255:red, 74; green, 144; blue, 226 }  ,draw opacity=1 ]   (507.36,309.91) -- (523.83,309.91) ;
\draw [color={rgb, 255:red, 74; green, 144; blue, 226 }  ,draw opacity=1 ]   (523.73,337.95) -- (538.5,323.63) ;
\draw [color={rgb, 255:red, 74; green, 144; blue, 226 }  ,draw opacity=1 ]   (523.83,309.91) -- (537.59,323.27) ;
\draw [color={rgb, 255:red, 74; green, 144; blue, 226 }  ,draw opacity=1 ]   (506.09,271.29) -- (522.56,271.29) ;
\draw [color={rgb, 255:red, 74; green, 144; blue, 226 }  ,draw opacity=1 ]   (505.58,284.04) -- (522.05,284.04) ;
\draw [color={rgb, 255:red, 74; green, 144; blue, 226 }  ,draw opacity=1 ]   (522.05,284.04) -- (537.25,298.62) ;
\draw [color={rgb, 255:red, 74; green, 144; blue, 226 }  ,draw opacity=1 ]   (522.56,271.29) -- (539.03,271.29) ;
\draw [color={rgb, 255:red, 74; green, 144; blue, 226 }  ,draw opacity=1 ]   (539.03,271.29) -- (554.23,285.86) ;
\draw [color={rgb, 255:red, 74; green, 144; blue, 226 }  ,draw opacity=1 ]   (537.25,298.62) -- (552.46,313.19) ;
\draw [color={rgb, 255:red, 74; green, 144; blue, 226 }  ,draw opacity=1 ]   (552.46,313.19) -- (567.66,327.76) ;
\draw [color={rgb, 255:red, 74; green, 144; blue, 226 }  ,draw opacity=1 ]   (554.23,285.86) -- (568.67,286.23) ;
\draw [color={rgb, 255:red, 74; green, 144; blue, 226 }  ,draw opacity=1 ]   (568.67,286.23) -- (583.62,270.2) ;
\draw [color={rgb, 255:red, 74; green, 144; blue, 226 }  ,draw opacity=1 ]   (583.62,270.2) -- (600.09,270.2) ;
\draw [color={rgb, 255:red, 74; green, 144; blue, 226 }  ,draw opacity=1 ]   (567.66,327.76) -- (584.13,327.76) ;
\draw [color={rgb, 255:red, 74; green, 144; blue, 226 }  ,draw opacity=1 ]   (584.13,327.76) -- (600.17,327.97) ;
\draw [color={rgb, 255:red, 74; green, 144; blue, 226 }  ,draw opacity=1 ]   (583.37,293.27) -- (599.5,293.64) ;
\draw [color={rgb, 255:red, 74; green, 144; blue, 226 }  ,draw opacity=1 ]   (583.62,310.76) -- (600.09,310.76) ;
\draw [color={rgb, 255:red, 74; green, 144; blue, 226 }  ,draw opacity=1 ]   (566.9,303.47) -- (583.37,293.27) ;
\draw [color={rgb, 255:red, 74; green, 144; blue, 226 }  ,draw opacity=1 ]   (566.9,303.47) -- (583.62,310.76) ;
\draw [color={rgb, 255:red, 74; green, 144; blue, 226 }  ,draw opacity=1 ]   (474.67,287.32) -- (491.14,287.32) ;
\draw [color={rgb, 255:red, 74; green, 144; blue, 226 }  ,draw opacity=1 ]   (474.17,300.07) -- (490.64,300.07) ;
\draw [color={rgb, 255:red, 74; green, 144; blue, 226 }  ,draw opacity=1 ]   (474.42,309.91) -- (490.89,309.91) ;
\draw [color={rgb, 255:red, 74; green, 144; blue, 226 }  ,draw opacity=1 ]   (474.42,323.02) -- (490.89,323.02) ;
\draw [color={rgb, 255:red, 74; green, 144; blue, 226 }  ,draw opacity=1 ]   (490.89,309.91) -- (507.36,309.91) ;
\draw [color={rgb, 255:red, 74; green, 144; blue, 226 }  ,draw opacity=1 ]   (401.09,270.7) .. controls (402.92,269.11) and (404.62,269.19) .. (406.2,270.96) .. controls (407.65,272.75) and (409.35,272.9) .. (411.32,271.42) .. controls (413.29,269.99) and (414.98,270.23) .. (416.4,272.13) .. controls (417.66,274.05) and (419.2,274.34) .. (421.01,273.01) .. controls (423.08,271.78) and (424.72,272.19) .. (425.91,274.23) .. controls (426.96,276.27) and (428.56,276.76) .. (430.73,275.7) .. controls (432.91,274.69) and (434.52,275.25) .. (435.55,277.4) .. controls (436.52,279.55) and (438.07,280.14) .. (440.22,279.17) .. controls (442.37,278.2) and (443.93,278.79) .. (444.92,280.94) .. controls (445.79,283.02) and (447.33,283.55) .. (449.54,282.54) .. controls (451.53,281.42) and (453.08,281.88) .. (454.18,283.93) .. controls (455.43,285.98) and (457.12,286.39) .. (459.25,285.17) .. controls (461.12,283.85) and (462.71,284.15) .. (464.04,286.08) .. controls (465.53,287.99) and (467.13,288.22) .. (468.86,286.77) .. controls (470.91,285.31) and (472.68,285.48) .. (474.17,287.28) -- (474.67,287.32) ;
\draw [color={rgb, 255:red, 74; green, 144; blue, 226 }  ,draw opacity=1 ]   (400.5,294.14) .. controls (402.48,292.41) and (404.25,292.37) .. (405.82,294.01) .. controls (407.29,295.66) and (408.87,295.64) .. (410.57,293.96) .. controls (412.51,292.31) and (414.21,292.33) .. (415.68,294.02) .. controls (417.29,295.73) and (418.89,295.79) .. (420.5,294.2) .. controls (422.29,292.65) and (424,292.77) .. (425.62,294.58) .. controls (427.03,296.41) and (428.66,296.6) .. (430.52,295.17) .. controls (432.51,293.8) and (434.18,294.08) .. (435.53,296.02) .. controls (436.8,297.96) and (438.41,298.27) .. (440.34,296.95) .. controls (442.32,295.63) and (443.97,295.93) .. (445.29,297.85) .. controls (446.8,299.78) and (448.52,300.03) .. (450.46,298.61) .. controls (452.26,297.14) and (453.82,297.31) .. (455.15,299.12) .. controls (456.92,300.95) and (458.61,301.08) .. (460.2,299.52) .. controls (462.17,297.97) and (463.96,298.07) .. (465.55,299.82) .. controls (466.92,301.54) and (468.57,301.6) .. (470.52,299.99) -- (474.17,300.07) ;
\draw [color={rgb, 255:red, 74; green, 144; blue, 226 }  ,draw opacity=1 ]   (401.09,311.26) .. controls (403.15,309.79) and (404.86,309.96) .. (406.22,311.76) .. controls (407.84,313.55) and (409.53,313.67) .. (411.29,312.1) .. controls (412.92,310.51) and (414.58,310.57) .. (416.26,312.28) .. controls (417.93,313.95) and (419.55,313.95) .. (421.12,312.26) .. controls (422.78,310.54) and (424.48,310.47) .. (426.21,312.05) .. controls (427.95,313.6) and (429.64,313.46) .. (431.27,311.65) .. controls (432.64,309.84) and (434.28,309.67) .. (436.19,311.15) .. controls (437.92,312.64) and (439.59,312.46) .. (441.2,310.61) .. controls (442.68,308.79) and (444.35,308.63) .. (446.2,310.14) .. controls (447.87,311.68) and (449.46,311.57) .. (450.96,309.81) .. controls (452.59,308.06) and (454.34,307.98) .. (456.2,309.58) .. controls (457.84,311.21) and (459.5,311.19) .. (461.17,309.52) .. controls (462.98,307.86) and (464.66,307.88) .. (466.2,309.58) .. controls (467.77,311.29) and (469.44,311.34) .. (471.2,309.75) -- (474.42,309.91) ;
\draw [color={rgb, 255:red, 74; green, 144; blue, 226 }  ,draw opacity=1 ]   (401.17,328.47) .. controls (402.78,326.57) and (404.55,326.34) .. (406.49,327.77) .. controls (408.32,329.22) and (409.92,329.03) .. (411.27,327.2) .. controls (412.85,325.35) and (414.55,325.18) .. (416.36,326.67) .. controls (418.05,328.19) and (419.66,328.05) .. (421.21,326.25) .. controls (422.68,324.48) and (424.31,324.37) .. (426.11,325.94) .. controls (427.96,327.52) and (429.6,327.45) .. (431.03,325.74) .. controls (432.8,324.03) and (434.48,324) .. (436.09,325.64) .. controls (437.86,327.28) and (439.53,327.26) .. (441.11,325.57) .. controls (442.72,323.86) and (444.43,323.82) .. (446.22,325.44) .. controls (447.95,327.03) and (449.65,326.95) .. (451.34,325.2) .. controls (452.89,323.43) and (454.47,323.33) .. (456.1,324.89) .. controls (457.85,326.42) and (459.52,326.28) .. (461.09,324.47) .. controls (462.76,322.64) and (464.49,322.47) .. (466.27,323.96) .. controls (468.17,325.43) and (469.75,325.25) .. (471,323.43) -- (474.42,323.02) ;
\draw  [dash pattern={on 4.5pt off 4.5pt}]  (437.29,229.62) -- (437,373.33) ;
\draw [line width=0.75]  [dash pattern={on 0.84pt off 2.51pt}]  (403.21,381.21) -- (630,381.83) ;
\draw [shift={(632,381.83)}, rotate = 180.16] [color={rgb, 255:red, 0; green, 0; blue, 0 }  ][line width=0.75]    (7.65,-2.3) .. controls (4.86,-0.97) and (2.31,-0.21) .. (0,0) .. controls (2.31,0.21) and (4.86,0.98) .. (7.65,2.3)   ;
\draw [shift={(401.21,381.2)}, rotate = 0.16] [color={rgb, 255:red, 0; green, 0; blue, 0 }  ][line width=0.75]    (7.65,-2.3) .. controls (4.86,-0.97) and (2.31,-0.21) .. (0,0) .. controls (2.31,0.21) and (4.86,0.98) .. (7.65,2.3)   ;
\draw [line width=0.75]    (403.11,82.89) -- (403.11,90.56) ;
\draw [shift={(403.11,92.56)}, rotate = 270] [fill={rgb, 255:red, 0; green, 0; blue, 0 }  ][line width=0.08]  [draw opacity=0] (4.8,-1.2) -- (0,0) -- (4.8,1.2) -- cycle    ;
\draw  [dash pattern={on 0.84pt off 2.51pt}]  (37.17,211.5) -- (641,211.67) ;
\draw  [color={rgb, 255:red, 155; green, 155; blue, 155 }  ,draw opacity=0.36 ][fill={rgb, 255:red, 155; green, 155; blue, 155 }  ,fill opacity=0.46 ] (87.56,105.13) .. controls (87.56,95.98) and (94.97,88.57) .. (104.11,88.57) -- (180.72,88.57) .. controls (189.86,88.57) and (197.27,95.98) .. (197.27,105.13) -- (197.27,154.78) .. controls (197.27,163.92) and (189.86,171.33) .. (180.72,171.33) -- (104.11,171.33) .. controls (94.97,171.33) and (87.56,163.92) .. (87.56,154.78) -- cycle ;
\draw  [color={rgb, 255:red, 155; green, 155; blue, 155 }  ,draw opacity=0.36 ][fill={rgb, 255:red, 155; green, 155; blue, 155 }  ,fill opacity=0.46 ] (87.56,276.13) .. controls (87.56,266.98) and (94.97,259.57) .. (104.11,259.57) -- (180.72,259.57) .. controls (189.86,259.57) and (197.27,266.98) .. (197.27,276.13) -- (197.27,325.78) .. controls (197.27,334.92) and (189.86,342.33) .. (180.72,342.33) -- (104.11,342.33) .. controls (94.97,342.33) and (87.56,334.92) .. (87.56,325.78) -- cycle ;
\draw [color={rgb, 255:red, 208; green, 2; blue, 27 }  ,draw opacity=1 ] [dash pattern={on 4.5pt off 4.5pt}]  (401.09,270.2) .. controls (447.67,240.5) and (447.17,274.5) .. (400.5,293.64) ;
\draw [color={rgb, 255:red, 208; green, 2; blue, 27 }  ,draw opacity=1 ] [dash pattern={on 4.5pt off 4.5pt}]  (401.09,310.76) .. controls (445.67,280) and (447.17,312) .. (401.17,327.97) ;
\draw [color={rgb, 255:red, 74; green, 144; blue, 226 }  ,draw opacity=1 ]   (600.09,270.56) .. controls (646.67,240.86) and (646.17,274.86) .. (599.5,293.99) ;
\draw [color={rgb, 255:red, 74; green, 144; blue, 226 }  ,draw opacity=1 ]   (600.09,311.11) .. controls (644.67,280.36) and (646.17,312.36) .. (600.17,328.33) ;

\draw [line width=0.75]  [dash pattern={on 0.84pt off 2.51pt}]  (204.71,365.96) -- (431.5,366.58) ;
\draw [shift={(433.5,366.58)}, rotate = 180.16] [color={rgb, 255:red, 0; green, 0; blue, 0 }  ][line width=0.75]    (7.65,-2.3) .. controls (4.86,-0.97) and (2.31,-0.21) .. (0,0) .. controls (2.31,0.21) and (4.86,0.98) .. (7.65,2.3)   ;
\draw [shift={(202.71,365.95)}, rotate = 0.16] [color={rgb, 255:red, 0; green, 0; blue, 0 }  ][line width=0.75]    (7.65,-2.3) .. controls (4.86,-0.97) and (2.31,-0.21) .. (0,0) .. controls (2.31,0.21) and (4.86,0.98) .. (7.65,2.3)   ;

\draw (309.21,367.75) node [anchor=north west][inner sep=0.75pt]   [align=left] {{\fontfamily{ptm}\selectfont $\mathcal{R}$}};

\draw (123.31,66.45) node [anchor=north west][inner sep=0.75pt]   [align=left,xshift = -10mm] {{\fontfamily{ptm}\selectfont $B_u\cup \cdots\cup B_{u+t-1}$}};
\draw (309.21,44) node [anchor=north west][inner sep=0.75pt]   [align=left] {{\fontfamily{ptm}\selectfont $\mathcal{R}$}};
\draw (412.77,178.85) node [anchor=north west][inner sep=0.75pt]   [align=left] {{\fontfamily{ptm}\selectfont $\mathcal{R}'$}};
\draw (401.02,66.73) node [anchor=north west][inner sep=0.75pt]   [align=left] {$I$};
\draw (322.31,65.95) node [anchor=north west][inner sep=0.75pt]   [align=left,xshift = -10mm,yshift = -0.5mm] {{\fontfamily{ptm}\selectfont $B_v\cup\cdots\cup B_{v+t-1}$}};
\draw (123.31,236.95) node [anchor=north west][inner sep=0.75pt]   [align=left,xshift = -10mm] {{\fontfamily{ptm}\selectfont $B_u\cup\cdots\cup B_{u+t-1}$}};
\draw (322.31,236.45) node [anchor=north west][inner sep=0.75pt]   [align=left,,xshift = -10mm,yshift = -0.5mm] {{\fontfamily{ptm}\selectfont $B_v\cup\cdots\cup B_{v+t-1}$}};
\draw (521.31,236.45) node [anchor=north west][inner sep=0.75pt]   [align=left] {};
\draw (513.21,363.17) node [anchor=north west][inner sep=0.75pt]   [align=left] {{\fontfamily{ptm}\selectfont $\mathcal{R}^g$}};
\draw (399.78,97) node [anchor=north west][inner sep=0.75pt]   [align=left,yshift = -0.5mm] {{\fontfamily{ptm}\fontsize{5}{6}\selectfont $x_1$}};
\draw (399.44,116.67) node [anchor=north west][inner sep=0.75pt]   [align=left,yshift = 1.5mm] {{\fontfamily{ptm}\fontsize{5}{6}\selectfont $x_2$}};
\draw (399.44,134.67) node [anchor=north west][inner sep=0.75pt]   [align=left,yshift = 2mm] {{\fontfamily{ptm}\fontsize{5}{6}\selectfont $x_3$}};
\draw (399.78,152) node [anchor=north west][inner sep=0.75pt]   [align=left,yshift = 1.5mm] {{\fontfamily{ptm}\fontsize{5}{6}\selectfont $x_4$}};
\draw (122.5,176) node [anchor=north west][inner sep=0.75pt]   [align=left] {$E_u$};
\draw (325,178) node [anchor=north west][inner sep=0.75pt]   [align=left] {$E_v$};
\draw (120.5,348.83) node [anchor=north west][inner sep=0.75pt]   [align=left] {$E_u$};
\draw (321.5,347.83) node [anchor=north west][inner sep=0.75pt]   [align=left] {$E_v$};

\draw    (462.3,86.33) .. controls (484.33,85.03) and (465.33,129.53) .. (483.33,130.53) ;
\draw    (483.33,130.53) .. controls (465.33,130.03) and (484.45,176.29) .. (462.45,176.79) ;

\draw (490,125.33) node [anchor=north west][inner sep=0.75pt]   [align=left] {The cycle $X$};
\end{tikzpicture}
    \caption{Finding a longer cycle than the given cycle $X$ in vertex-transitive graphs $G$.}
    \label{fig:thm9}
\end{figure}

To conclude this section, we derive Corollary~\ref{cor:vert-tran} from Theorem~\ref{thm:transitive babai}.

\begin{proof}[\bf Proof of Corollary~\ref{cor:vert-tran}.]
    By Corollary~\ref{cor:smith} and Theorem~\ref{thm:transitive babai}, every two longest cycles in a connected $d$-regular vertex-transitive graph on $n$ vertices intersect in at least 
    $$\min_{d\geq2}\max\left\{\Omega\big((\log_dn)^{1/3}\big),\Omega(d^{2/3})\right\}=\Omega\big((\ln n/\ln\ln n)^{1/3}\big)$$ vertices, where the minimum is achieved when $d=\Theta\big((\ln n/\ln\ln n)^{1/2}\big)$. 
\end{proof}

\section{Concluding remarks}
In this paper, we obtain two results on the intersection of longest cycles in graphs $G$, which are closely related to the problems of estimating the size of transversals and the parameter $c(G)$.

In the proof of Theorem~\ref{thm:transitive babai}, given a cycle $X$, we exploit the “block structure” of vertex-transitive graphs to construct a longer cycle by assembling certain subpaths of $X$ and their images under automorphisms.
It would be interesting to see whether this approach can be further improved and applied in other contexts.

Next we discuss a natural extension of longest cycles to the so-called \emph{maximum $R$-subdivisions}, a concept studied in~\cite{long2021sublinear, kierstead2023improved}, and we present a corresponding generalization of Theorem~\ref{thm:lct upp}. 

\begin{define}\label{def:R-subdiv}
Let $R$ be a multigraph.
A {\bf maximum $R$-subdivision} in a graph $G$ is a subdivision of $R$ with the maximum number of vertices in $G$. 
\end{define}

Let $P_2$ be a single edge and $C_1$ a single vertex with a loop. 
Then a longest path (respectively, cycle) in a graph $G$ is a maximum $P_2$-subdivision (respectively, $C_1$-subdivision) in $G$.

The following result extends Theorem~\ref{thm:lct upp} to maximum $R$-subdivisions for any multigraph $R$. 
Notably, the analogous statement of Theorem~\ref{thm:lct upp} also holds for longest paths.

\begin{thm}\label{thm:sep size upper bound subdiv}
For any multigraph $R$, if two maximum $R$-subdivisions in a graph $G$ share $m$ vertices, then there exists a vertex cut of size $O_R(m^\frac{3}{2})$ separating them.
\end{thm}

\begin{proof}[Proof sketch]
    Let $X$ and $Y$ be two maximum $R$-subdivisions in $G$, and define $M = V(X) \cap V(Y)$ and $m = |M|$.
For each $e \in E(R)$, let $P_e(X)$ denote the subdivided path or cycle in $X$ corresponding to $e$ -- specifically, a path if $e = P_2$, or a cycle if $e = C_1$.

By the pigeonhole principle, among the $e(R)^2$ edge pairs $(e, e') \in E(R)^2$, any collection $\mathcal{P}$ of $k$ disjoint $(X-M,Y-M)$-paths must include at least $k/e(R)^2$ paths whose endpoints lie in $P_e(X)$ and $P_{e'}(Y)$, for some pair $(e, e') \in E(R)^2$.
Decompose $P_e(X) - V(Y)$ (resp. $P_{e'}(Y) - V(X)$) into segments $X_1, \ldots, X_m$ (resp. $Y_1, \ldots, Y_m$)\footnote{If $e$ or $e'$ is $P_1$, the number of segments becomes $m+1$.}, and construct the auxiliary bipartite graph $F := F(P_e(X), P_{e'}(Y), \mathcal{P})$ as in Definition~\ref{def:aux F}. Here, each vertex of $F$ represents a segment $X_i$ or $Y_j$, and $F$ contains at least $k/e(R)^2$ edges, each representing an $(X-M, Y-M)$-path whose endpoints lie in the corresponding segments.

We assert that $e(F) = O_R(m^{3/2})$. 
The proof largely follows Section~3 (which covers the case $R = C_1$).
Here, we focus on $e = e' = P_2$, highlighting only the differences and providing detailed explanations as needed; 
the other cases where $e$ or $e'$ is $C_1$ can be handled similarly.

    In this case, $P_e(X)$ and $P_{e'}(Y)$ are two paths in $G$, and the definition of a winning certificate should be adapted to a pair $(Q_1, Q_2)$ of paths satisfying the following conditions:
\begin{itemize}
\item $Q_1$ has the same endpoints as $P_{e'}(Y)$, and $Q_2$ has the same endpoints as $P_e(X)$.
\item $V(Q_1) \cap V(Y) \subseteq V(P_{e'}(Y))$, and $V(Q_2) \cap V(X) \subseteq V(P_e(X))$.
\item $E(Q_1) \cup E(Q_2) \supseteq E(P_e(X)) \cup E(P_{e'}(Y))$, and $|Q_1| + |Q_2| > |P_e(X)| + |P_{e'}(Y)|$.
\end{itemize}
    Given such a winning certificate, replacing $P_e(X)$ by $Q_2$ in $X$ and $P_{e'}(Y)$ by $Q_1$ in $Y$ produces two $R$-subdivisions $X'$ and $Y'$ such that $e(X') + e(Y') > e(X) + e(Y)$, contradicting the maximality of $X$ and $Y$.
    The red and blue paths in Figure~\ref{fig:type00} and Figure~\ref{fig:both} still serve as examples of such winning certificates.
    The main difference, compared to the proof in Section 3, arises in the proof of Lemma~\ref{lem:crossing int}, where we can no longer assume that the $Y$-segments are arranged as in Figure~\ref{fig:both} and there would be 32 cases rather than 4 cases to be verified.
    To see this, 
    in addition to the configuration in Figure~\ref{fig:both}, the segments on $Y$ may also be arranged as $(Y_{k_1}, Y_{\ell_1}, Y_{\ell_2}, Y_{k_2})$.
    For each 4-cycle, there are two possible relative orderings of the two vertex pairs on the $X$-segments (and similarly for the $Y$-segments).
    This gives four distinct relative orderings per 4-cycle.
    In total, there are $2 \cdot 4^2 = 32$ possible configurations to consider.
We omit the detailed verification that, using constructions almost identical to Lemma~\ref{lem:crossing int}, each case indeed produces a winning certificate.

It follows that $k/e(R)^2\leq e(F)=O(m^{3/2})$, which implies $k=O_R(m^{3/2})$.
The conclusion then follows by Menger's Theorem.
\end{proof}

\bibliographystyle{abbrv}
\bibliography{ref}

\end{document}